\documentclass[a4paper]{article}
\usepackage{float}
\usepackage[english]{babel}
\usepackage{amsmath}
\usepackage{graphicx}
\usepackage{epstopdf}
\usepackage{epsfig}
\usepackage{amssymb}
\usepackage[colorinlistoftodos]{todonotes}
\usepackage{amsthm}
\def\Res{\mathop{\rm Res}}
\usepackage{comment}
\usepackage{tikz}
\usetikzlibrary{arrows,
decorations.pathreplacing,decorations.markings}
\tikzset{
  on each segment/.style={
    decorate,
    decoration={
      show path construction,
      moveto code={},
      lineto code={
        \path [#1]
        (\tikzinputsegmentfirst) -- (\tikzinputsegmentlast);
      },
      curveto code={
        \path [#1] (\tikzinputsegmentfirst)
        .. controls
        (\tikzinputsegmentsupporta) and (\tikzinputsegmentsupportb)
        ..
        (\tikzinputsegmentlast);
      },
      closepath code={
        \path [#1]
        (\tikzinputsegmentfirst) -- (\tikzinputsegmentlast);
      },
    },
  },
  mid arrow/.style={postaction={decorate,decoration={
        markings,
        mark=at position .5 with {\arrow[#1]{stealth}}
      }}},
}

\usepackage[top=3cm,bottom=3cm,left=2cm,right=2cm]{geometry}
\usepackage{enumerate}
\usepackage{graphicx}
\usepackage{amsmath}

\numberwithin{equation}{section}
\newtheorem{theorem}{Theorem}[section]

\newtheorem{remark}[theorem]{Remark}
\newtheorem{prop}[theorem]{Proposition}

\def \a{\alpha}

\def\Re{\text{\upshape Re\,}}
\def\Im{\text{\upshape Im\,}}
\def\Ai{\text{\upshape Ai}}
\def\Bi{\text{\upshape Bi}}

\usetikzlibrary{calc}

\def\(#1\){\left(#1\right)}

\begin{document}
\title{Uniform asymptotics for discrete orthogonal  polynomials on infinite nodes with an accumulation point}
\author{Xiao-Bo Wu$^{\,1}$, Yu Lin$^{\,2,}$\thanks{Corresponding author. Email address: scyulin@scut.edu.cn} , Shuai-Xia Xu$^{\,3}$, Yu-Qiu Zhao$^{\,1}$
\\
 \hbox{\small \emph{$^1$ Department of Mathematics, Sun Yat-sen University, Guangzhou,
China}} \\
 \hbox{\small
\emph{$^2$ Department of Mathematics, South China University of Technology, Guangzhou, China}
 }
\\
 \hbox{\small
\emph{$^3$ Institut Franco-Chinois de I'Energie Nucl\'eaire, Sun Yat-sen University, Guangzhou,
China}}}












\date{}

\maketitle

\begin{abstract}
In this paper, we 
develop the Riemann-Hilbert method to study the 
asymptotics of discrete orthogonal  polynomials on infinite nodes with an accumulation point. 
To illustrate our method, we consider the Tricomi-Carlitz polynomials $f_n^{(\alpha)}(z)$ where $\alpha$ is a positive parameter.
Uniform Plancherel-Rotach type asymptotic formulas are obtained in the entire complex plane
including a neighborhood of the origin, and
our results agree with the ones obtained earlier in [{\it SIAM J.\;Math.\;Anal} {\bf 25} (1994)] and 
[{{\it Proc.\;Amer.\;Math.\;Soc.\,}{\bf138} (2010)}].  

\vspace{.4cm}


\end{abstract}


\section{Introduction}

Discrete orthogonal polynomials arise in  many fields of  mathematical physics, such as random matrix theory and quantum mechanics.
There has been a considerable amount of interest in the  asymptotic analysis of   these
  orthogonal polynomials, and various  methods have been  developed for this purpose; see \cite{Szego} and \cite{WongZhao}.

In 2007, Baik {\it et al.}\;\cite{Baik}    studied the asymptotics of discrete orthogonal polynomials with respect to a general weight function by using the Riemann-Hilbert approach. The starting point of their  investigation  is  the  interpolation problem (IP)  for    discrete orthogonal polynomials, introduced in  Borodin and Boyarchenko   \cite{Borodin}.  The IP is then turned into a Riemann-Hilbert problem (RHP), and the Deift-Zhou method for oscillating RHP   applies; see, e.g., \cite{DeiftZhouA,DeiftZhouU, DeiftZhouS}. The work of    Baik {\it et al.} furnishes  an important step forward of the method of Deift and Zhou.

Much attention has been attracted lately.  For example,  Wong  and coworkers considered  cases with finite nodes \cite{DaiWong,LinWong,LinWongC},  and infinite nodes \cite{OuWong,WangWong}  regularly distributed. A common feature in these work is termed  global asymptotics, with global referring to the   domains of uniformity.

Quite recently,
Bleher and Liechty \cite{BleherE,Bleher} made a major modification to the method in the treatment of the so-called band-saturated region endpoints, when they were
 considering  the large-$N$ asymptotics of a system of discrete orthogonal polynomials
with respect to the varying exponential weight $e^{-NV(x)}$ on the regular infinite lattice of mesh $1/N$,
where $V(x)$ is a real analytic function with sufficient growth at infinity.  Here  regular infinite lattice means that the infinite nodes are equally spaced.

In this paper, we study the asymptotics of discrete orthogonal polynomials on infinite nodes with an accumulation point. We illustrate our method by
concentrating on the Tricomi-Carlitz polynomials. It is worth noting that there are other polynomials share such a structure. For example,  there is a class of sieved Pollaczek polynomials defined by a second-order difference equation; see
 \cite{WangZhao}.  A significant fact is  that the corresponding orthogonal measure  consists of an absolutely continuous part, and a discrete part having infinite many mass points with an accumulation point.

The Tricomi-Carlitz polynomials
are also  of interest. Initially,
Tricomi \cite{Tricomi}
introduced a class of non-orthogonal polynomials
   $t_n^{(\a)}(x)$,  related to the Laguerre polynomials via  $ t_n^{(\a)}(x)=(-1)^nL_n^{(x-\alpha-n)} (x)$. They are  explicitly given by
\begin{equation}\label{t}
 t_n^{(\a)}(x)=\sum_{k=0}^{n}(-1)^k \begin{pmatrix}
 x-\a\\
 k
 \end{pmatrix}
\frac {x^{n-k}}{(n-k)!},\quad\quad n=0,1,2,\cdots;
\end{equation}cf. \cite{LeeWongU,LopezTemme}. We note that each   $t_n^{(\a)}(x)$ is of degree $[\frac n 2]$, and the polynomials satisfy
  the following  recurrence relation
\begin{equation}
	(n+1) t_{n+1}^{(\alpha)}
	(x)- (n+\alpha) t_n^{(\alpha)}(x)
	+x t_{n-1}^{(\alpha)}(x)=0,
	\qquad n\geq 1,
\end{equation}
with initial values $t_0^{(\alpha)}(x)=1$ and
$t_{1}^{(\alpha)}(x)=\alpha$.
Carlitz  \cite{Carlitz}  revisited these polynomials, and found that if one set
\begin{equation}
 f_n^{(\a)} (x)= x^n t_n^{(\a)}(x^{-2}),
\end{equation}
then $f_n^{(\a)} (x)$ possess  the following orthogonality
\begin{equation}
 	\int_{-\infty}^{\infty}  f_m^{(\alpha)} (x)
    f_n^{ (\alpha) }(x) d \psi^{(\alpha)} (x)
    =h_n \delta_{mn},\quad\quad h_n=\frac{ 2e^{\alpha} }{ (n+\alpha) n! },
\label{hn}
\end{equation}
where $\alpha$ is a positive number, $\psi^{(\alpha)}(x)$ is the step function with  jumps
\begin{equation}\label{mass}
		d\psi^{(\alpha)} (x)
        =\frac{ (k+\alpha)^{k-1} e^{-k} }{k!}
        \qquad \mbox{at } x=\pm x_k,
\end{equation}
and the nodes  $x_k=  (k+\alpha)^{-1/2}$ for $k=0,1,2,\cdots$. It is readily verified that  the Tricomi-Carlitz polynomials satisfy the recurrence relation
\begin{equation}\label{t1.6}
	(n+1) f_{n+1}^{(\alpha)}
	(x)- (n+\alpha) x f_n^{(\alpha)}(x)
	+ f_{n-1}^{(\alpha)}(x)=0,
	\qquad n\geq 1,
\end{equation}
with initial values $f_0^{(\alpha)}(x)=1$, and
$f_{1}^{(\alpha)}(x)=\alpha x$.   From \eqref{t1.6} we see that a  symmetry relation holds, namely,
\begin{equation}
f_{n}^{(\alpha)}(z)=(-1)^nf_{n}^{(\alpha)}(-z).
\label{sym}
\end{equation}Also, if
  we denote the monic polynomials by
\begin{equation}
	\pi_n(z):= f_n^{(\alpha)}(x)
	/\gamma_n,
\end{equation}
then  the leading coefficient of
$f_n^{(\alpha)}(z)$ is
\begin{equation}
	\gamma_n=\prod_{k=0}^{n-1}
\frac{k+\a}{k+1}
=\frac{\Gamma(n+\a)}{\Gamma(\a)\Gamma(n+1)}.
\label{leading}
\end{equation}
Moreover, the Tricomi-Carlitz polynomials are also related to the random walk polynomials $r_n(x,\alpha)$,
which were
discovered by Karlin and McGregor in \cite{Karlin} to the study of a birth and death process.
For more information on orthogonal polynomials, we refer to \cite{Szego} and \cite{BealsWong}.

Asymptotic behavior of these polynomials  was first investigated by Goh and Wimp in  \cite{GohWimpO} for    $f_n^{(\alpha)}(y/\sqrt \alpha)$,   and  in \cite{GohWimpT} for    $f_n^{(\alpha)}(y/\sqrt n)$.  Later,  L\'{o}pez  and  Temme  \cite{LopezTemme} took  $f_n^{(\alpha)}(x)$  as an example  to  approximate   polynomials in terms of the Hermite polynomials.
 The present paper is also inspired by the work of    Lee and Wong. In \cite{LeeWongU}, Lee and Wong
    derived an asymptotic expansion for $f_n^{(\alpha)}(t/\sqrt{\nu})$
   by using the difference equation method,
which holds  for $t$ in $[0,\infty)$, where $\nu = n +2\alpha - 1/2$. In another paper \cite{LeeWongA}, an alternative   integral method is used to  derive  the expansion, along with asymptotic formulas for  the extreme zeros.
In both treatments they obtained the uniform asymptotics,  in terms of the Airy function,  at and around the turning point $t=2$.

It is worth noting that the previous uniform results  are obtained on the real line, while an advantage of the Riemann-Hilbert approach lies in that the uniform  asymptotic approximations can be obtained in overlapping domains covering the whole complex plane.  In the present case,  the nodes are not regularly distributed, with the origin being the accumulation point. Even worse, the mass in   \eqref{mass}
also   shows a singularity at the origin. Hence, it is of interest to see the influence of these singularities on the asymptotic behavior at the origin.

The objective of this paper is to derive the uniform asymptotics of the Tricomi-Carlitz polynomials,  based only on the weight function, and using several of the techniques    developed by
Baik {\it et al.} and
Bleher and Liechty.
The rest of this paper is arranged as follows:
In Section 2, we state our main results, and introduce the basic interpolation problem.
In Section 3, we will consider the transformation, which converts the interpolation problem into an  equivalent Riemann-Hilbert problem (RHP).
In Section 4, we give the matrix transformation $V(z)$ that normalizes the RHP for $R(z)$ presented in Section 3 by using a function
$g(z)$, which
is related to the logarithmic potential of the equilibrium measure.
Several  auxiliary functions, such as the
$\phi$-functions and $D$-functions,  are also studied in Section 4. Then, we factorize
the jump matrix in the RHP for $V(z)$ and construct the global parametrix $N(z)$ {in Section 5}.
In Section 6, we study  local parametrices and the last transformation $V \to S$.
The proofs of Theorem \ref{theorem-main} and {Theorem \ref{Theorem-origin}} are presented in {Section 7 and 8}, 
and we  compare our formulas with previous results in {Section 9}.

\section{Statement of Results}

It is well-known that the zero distribution
plays an important role in the asymptotic
analysis of the polynomials;
see \cite{DeiftZhouS} and \cite{Baik}.
Note that the asymptotic zero distribution of the Tricomi-Carlitz polynomials
is the Dirac point mass at zero; see Goh and Wimp \cite{GohWimpO}.
In the present paper, we study the large-$n$  asymptotics of
the rescaled Tricomi-Carlitz polynomials $f_n^{(\alpha)}(n^{-1/2}z)$,
of which  the density function has already been
  given as
\begin{equation}
\psi(x)=
\begin{cases}
\dfrac{1}{\pi} \left( \dfrac{ 4\arctan(|x|/\sqrt{4-x^2}) }{|x|^3}
- \dfrac{\sqrt{4-x^2}}{x^2}
\right),   & \quad  |x|\leq 2, \\[.4cm]
\dfrac{2}{|x|^3},  & \quad  |x|>2;
\end{cases}
\label{psi}
\end{equation}
see \cite{GohWimpT, ArnoAssche}.

To state the asymptotic behavior of the rescaled  polynomials, we need to introduce some  notations. Let the $g$-function be the logarithmic potential defined by
\begin{equation}
 g(z):= \int_{-\infty}^{\infty} \log(z-s) \psi(s) ds
\qquad
\mbox{for }  z\in\mathbb{C}\setminus \mathbb{R},
 \label{g}
\end{equation}with  the branch   chosen such that $\arg (z-s)\in (-\pi, \pi)$,
and the so-called $\phi$-function be given by
\begin{equation}
\begin{aligned}
	\phi(z):= l/2-g(z)
		\qquad
		\mbox{for }
		z \in \mathbb{C}_\pm,
\end{aligned}
\label{phi}
\end {equation}
where $l:=2\int_{-\infty}^{\infty} \log|2-s| \psi(s) ds$ is  the Lagrange multiplier; cf.   \eqref{l} below.
Also, we introduce the auxiliary function
\begin{equation}
	\widetilde\phi(z)
	:=	\int_2^{z}\left(
		-g'(s)\mp \frac {2\pi i}{s^3}
		\right)ds
	\qquad
	\mbox{for }
	z\in C_\pm,
\label{phi-tilde}
\end {equation}
where the path of integration  lies entirely in the regions $z\in \mathbb{C} \setminus (-\infty,2]$  except for the initial point.
In {Section 6}, we will show that the function
\begin{equation}\label{f-tilde}
	\widetilde f_n(z) := \( -\frac32 n \widetilde
	\phi (z) \)^{2/3}
	 \end{equation} is  analytic   in a neighborhood of  $z=2$.  The last function   we  need is
\begin{equation}\label{D}
D(z):= \frac{\Gamma(\alpha-n/z^2)}{e^{n/z^2}\sqrt{2\pi}}\left(-\frac{n}{z^2}\right)^{n/z^2-\alpha+1/2}\qquad
\mbox{for } z\in \mathbb{C}\setminus \mathbb{R},
\end{equation}
{
where the branch is  chosen such that  $-1/z^2=e^{\pm \pi i}/z^2$ for $z\in \mathbb{C}_\pm$, with $\arg z\in (-\pi, \pi)$.}

Now we are ready to present our main results.
In view of the symmetries \eqref{sym} and $f^{(\alpha)}_n(z)=\overline {f^{(\alpha)}_n(\overline{z})}$, we only need to present the  asymptotic formulas for $\pi_{n}(n^{-1/2} z)$    in the first quadrant of the complex plane. The asymptotic formulas are stated
in the five closed regions $A_{\delta}$, $B_{\delta}$,
$C_{\delta,1}$, $C_{\delta,2}$ and $D_\delta$  depicted in Figure \ref{fig-results}.

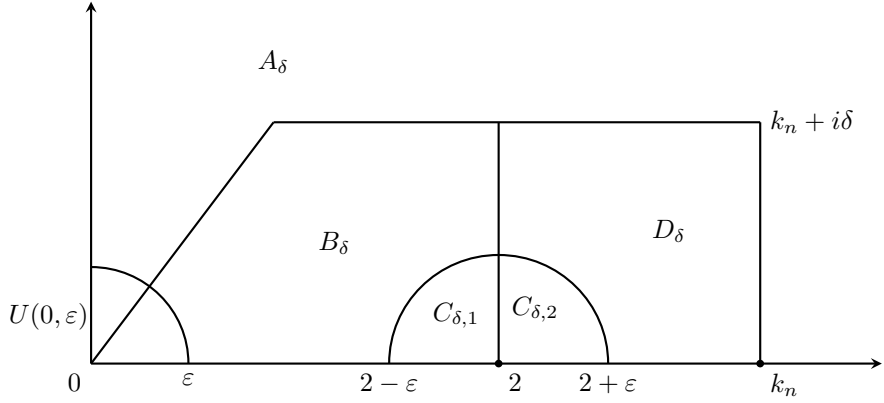
\begin{figure}[t]
\centering
\begin{tikzpicture}[auto,scale=.8,thick]

   \draw [->,>= stealth](-5,0) --(8,0);

   \fill (3.5-1.8,0) circle (1.8pt);

   \fill (6,0) circle (1.8pt);

      \draw [->,>= stealth](-5,0) --(-5,6);

   \draw (3.5,0)  arc (0:180:1.8) ;

   \draw  (-5+1.6,0) arc (0:90:1.6);

   \node [below] at (3.5,0) {$2+\varepsilon$} ;

   \node [below] at (3.5-3.6,0) {$2-\varepsilon$} ;

   \node [below] at (-5+1.6,0) {$\varepsilon$} ;

   \node [above] at (-5-.7,.4) {$U(0,\varepsilon)$} ;

   \draw  (3.5-1.8,0)--(3.5-1.8,4);

   \draw (-5,0)--(-2,4);

  \draw(-2,4)--(6,4);

  \draw   (6,4)--(6,0);

  \node [below left] at (-5,0) {$0$};

  \node [below right] at (6,0) {$k_n$};

    \node [right] at (6,4) {$k_n+i\delta$};
  
  \node [below right] at (3.5-1.8,0) {$2$};

  \node at  (-2,5) {$A_\delta$} ;

  \node at (-1,2) {$B_\delta$};

  \node at (4.5,2.2) {$D_\delta$};

  \node at (1,.8) {$C_{\delta,1}$};
  \node at (2.3,.9) {$C_{\delta,2}$};

\end{tikzpicture}
\caption{Asymptotic regions for $\pi_n(n^{-1/2} z)$ in the first quadrant.}
\label{fig-results}
\end{figure}

\begin{theorem}\label{theorem-main}
Let $\alpha>0$ and $k_n=\sqrt{n/\alpha}+\delta$. Then there exists $\delta_0>0$ such that for all $0<\varepsilon<\delta\leq \delta_0$,  the following holds (see Figure \ref{fig-results}):
\begin{enumerate}[(a)]
\item For $z$ in the outside region $A_\delta\setminus
    U(0,\varepsilon)	$,  where $U(0,\varepsilon)$ is the disk of radius $\varepsilon$, centered at the origin: 
\begin{equation}
\begin{aligned}
\pi_n (n^{-1/2}z)&=\frac{\Gamma(\a)e^{n/2}}
{\sqrt{2\pi}n^{n/2+\alpha-1/2}}D^{-1}(z)(z^2-4)^{-1/4}
\\
&\quad \times \(\frac{z+\sqrt{z^2-4}}{2}\)^{2\a-1/2}e^{-n\phi(z)-\a\pi i+\pi i/2}
\left (1+O(1/n)\right ).
\label{A_delta}
\end{aligned}
\end{equation}

\item For $z$ in the region $B_\delta\setminus
    U(0,\varepsilon)	$:
\begin{equation}
\begin{aligned}
&\pi_n (n^{-1/2}z)=\frac{\Gamma(\a)e^{n/2}}
{\sqrt{2\pi}
n^{n/2+\alpha-1/2}
}(z^2-4)^{-1/4}
\\
&\quad\times\Bigg\{\left[\(\frac{z+\sqrt{z^2-4}}{2}\)^{2\a-1/2}e^{-n\phi(z)-\alpha\pi i+\pi i/2}+\(\frac{z-\sqrt{z^2-4}}{2}\)^{2\a-1/2}e^{n\phi(z)+\alpha\pi i}\right]
\left( 1+O(1 /n)\right )
\\
&\qquad\quad+
O(e^{-n \Re \phi}/n)
\Bigg\}.
\label{B_delta}
\end{aligned}
\end{equation}

\item For z in the Airy region $C_{\delta,1} \cup C_{\delta,2}$:
\begin{equation}\label{C_delta}
\begin{aligned}
&\pi_n (n^{-1/2}z)=\frac{\Gamma(\a)e^{n/2}}
{\sqrt{2 }n^{n/2+\alpha -1/2}}
\left[\mathbf{A}(z,n)\left(1+O\left( {1}/{n}\right)\right)+\mathbf{B}(z,n)\left(1+O\left({1}/{n}\right)\right)\right],
\end{aligned}
\end{equation}
where
\begin{equation}\label{A}
\begin{aligned}
\mathbf{A}(z,n)&=\left[
\left(\frac{z+\sqrt{z^2-4}}{2}\right)^{2\alpha-1/2}
-\left(\frac{z-\sqrt{z^2-4}}{2}\right)^{2\alpha-1/2}
\right](z^2-4)^{-1/4}(\widetilde f_n(z))^{-1/4}
\\
&~~~\times\left[\Ai'(\widetilde f_n(z))\cos(\alpha\pi-n\pi/z^2)+\Bi'(\widetilde f_n(z))\sin(\alpha\pi-n\pi/z^2)\right]
\end{aligned}
\end{equation}
and
\begin{equation}\label{B}
\begin{aligned}
\mathbf{B}(z,n)&=\left[
\left(\frac{z+\sqrt{z^2-4}}{2}\right)^{2\alpha-1/2}
+\left(\frac{z-\sqrt{z^2-4}}{2}\right)^{2\alpha-1/2}
\right](z^2-4)^{-1/4}(\widetilde f_n(z))^{1/4}\\
&~~~\times \left[\Ai(\widetilde f_n(z))\cos(\alpha\pi-n\pi/z^2)+\Bi(\widetilde f_n(z))\sin(\alpha\pi-n\pi/z^2)\right].
\end{aligned}
\end{equation}


\item For z in the region $D_\delta$:
\begin{equation}
\begin{aligned}
&\pi_n (n^{-1/2}z)=\frac{\Gamma(\a)e^{n/2}}
{\sqrt{2\pi}n^{n/2+\alpha-1/2}}\\
&\quad\times \left[D^{-1}(z)(z^2-4)^{-1/4}
\(\frac{z+\sqrt{z^2-4}}{2}\)^{2\a-1/2}
e^{-n\phi(z)-\a\pi i+\pi i/2}\(1+O\left( 1/{n}\right)\)
+O\(e^{n\Re \phi}\)
\right].
\label{D_delta}
\end{aligned}
\end{equation}



\end{enumerate}
\label{uniform expansion}
\end{theorem}

\begin{figure}
\centering
\begin{tikzpicture}
\def\rad{1.8cm}
  \path [thick,draw=black]
  (140:1.4) --(140:1.9*\rad) 
  (-40:1.4)--(-40:1.9*\rad) 
  (-140:1.9*\rad)
  --(-140:1.4)
  (40:1.9*\rad) 
  --(40:1.4)
  (-1.7*\rad,0)--(0,0)
  (0,0)--(1.9*\rad,0)
  ;

\filldraw  (0,0) circle (1pt) node [below] {$0$};

\draw [dashed] (0,0) circle (1.4cm);

\end{tikzpicture}

\caption{The contour near $z=0$.}
\label{figure-0-2}
\end{figure}
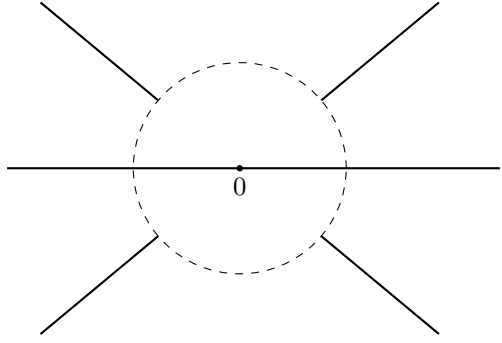

 In the disk $U(0, \varepsilon)$ of radius $ \varepsilon>0$, centered at the origin, we have the following 
  uniform asymptotic approximation:
     \begin{theorem}\label{Theorem-origin}
For $z\in U(0, \varepsilon)\cap \mathbb{C}_+$; see Figure \ref{figure-0-2},  it holds
 \begin{equation}\label{asy-origin}
\begin{aligned}
\pi_n (n^{-1/2}z)=&\frac {\Gamma(\alpha)}   {\sqrt{2\pi} }    n^{1/2-n/2-\alpha}   e^{n/2}  (4-z^2)^{-1/4}  \\
& \times \left [ e^{i\pi(1/4-\alpha)-n\phi(z)} \varphi(z/2)^{2\alpha-1/2} 
 \left(1+O(1/n)\right )
+e^{-i\pi(1/4-\alpha)+n\phi(z)} \varphi(z/2)^{-2\alpha+1/2} 
 \left(1+O(1/n)\right )
\right],
\end{aligned}
\end{equation}
  where branches are chosen such that $\arg (z+2)$, $\arg (2-z)$
and $\varphi(z)=z+\sqrt{z^2-1}$ is a analytic function in $\mathbb{C}\setminus [-1, 1]$ and behaves like $2z$ at infinity; see \eqref{varphi}.  The formula for $z\in U(0, \varepsilon)\cap \mathbb{C}_-$
is obtained from \eqref{asy-origin}    by taking complex conjugate.
\end{theorem}

To derive the main results, we begin with the  basic IP
 for the Tricomi-Carlitz polynomials.
 Following \cite{Baik}, one can formulate  the IP
 for a $2\times2$ matrix-value function $Y(z)$ with the properties:
\begin{description}
  \item[($Y_1$)]     $Y(z)$ is analytic in
  $\mathbb{C}\backslash\{0,  \pm x_{0},  \pm x_{1}, \cdots,  \pm x_{k}, \cdots \}$;

\item[($Y_2$)]    at each  $\pm x_k$, $k=0,1,\cdots$,
$Y(z)$ has a simple pole, and satisfies
\begin{equation}\label{Y_2}
 \Res_{z=\pm x_{ k}}Y(z)=\lim_{z\rightarrow \pm x_{ k}}Y(z) \left(
                               \begin{array}{cc}
                                 0&  w_d(z) \\
                                 0 & 0 \\
                                 \end{array}
                             \right)
,
\end{equation}
where
\begin{equation}
w_d(z)=\frac{({1}/{z^2})^{{1}/{z^2}-1-\alpha}~e^{-{1}/{z^2}+\alpha}}
{\Gamma({1}/{z^2}+1-\alpha)}
\label{w_d}
\end{equation}
and the branch is chosen such that 
$w_d(z)$ is analytic in $\mathbb{C}\setminus i \mathbb{R}$ and takes positive values for
$\mathbb{R}\setminus \{0\}$;
\item[($Y_3$)]    as $z \to \infty$,
  \begin{equation}\label{Y_3}
  Y(z)=\left (I+O\left (1 /z\right )\right )\left(
                               \begin{array}{cc}
                                 z^n & 0 \\
                                 0 & z^{-n} \\
                               \end{array}
                             \right);
\end{equation}
\item[($Y_4$)]
$Y(z)$ has the following behavior
as $z\to 0$,
 \begin{equation}\label{Y_4}
Y(z)=O
                               \begin{pmatrix}
                                 1 & \log |z| \\
                                 1 & \log |z| \\
                               \end{pmatrix},~~~\qquad z\not\in[-1/\sqrt{\alpha},1/\sqrt{\alpha}\,].\end{equation}
\end{description}

By the well-known theorem of Fokas, Its and Kitaev \cite{Fokas},
we have
\begin{theorem}\label{Theorem-Y}
 The unique solution of the IP for $Y$ is given by
\begin{equation}
  Y(z)=\left(
  \begin{array}{cc}
  \pi_n(z) & \displaystyle\sum_{k=0}^{\infty}
  	\frac{\pi_n(x_k)w_d(x_k)}{z-x_k}+\displaystyle\sum_{k=0}^{\infty}\frac{\pi_n(-x_k)w_d(-x_k)}{z+x_k}
  \\\\
  \gamma_{n-1}^2h_{n-1}^{-1}\pi_{n-1}(z) & \gamma_{n-1}^2h_{n-1}^{-1} \left(\displaystyle\sum_{k=0}^{\infty}
  	\frac{\pi_{n-1}(x_k)w_d(x_k)}{z-x_k}+\displaystyle\sum_{k=0}^{\infty}\frac{\pi_{n-1}(-x_k)w_d(-x_k)}{z+x_k}
  \right)
  \\
                                 \end{array}
                             \right),
\label{Y}
\end{equation}
where $\pi_n(z)$ is the monic Tricomi-Carlitz polynomials of degree $n$,  $\gamma_n$ and $h_n$
are defined in \eqref{leading} and \eqref{hn}, respectively.
\end{theorem}

\section{Riemann-Hilbert problem}

In this section, we will  make a sequence of transformations
$
	Y\to U\to R,
$
to convert  the basic IP  into a RHP.
The first transformation   is the    following rescaling of variable
\begin{equation}
  U(z):=n^{{n\sigma_3}/{2}}Y(n^{-1/2}z),
\label{Y to U}
\end{equation}
where $\sigma_3=\begin{pmatrix}
1 & 0\\
0 & -1
\end{pmatrix}$
is the Pauli matrix. Let $X$ denote the set defined by
\begin{equation}
 X:=\{\pm X_{ k}\}_{k=0}^{\infty},
 \quad \mbox{where }  X_{ k}= n^{1/2} \,x_{k}.
\end{equation}
The $\pm X_{ k}$'s are called {\it nodes}, and they all lie in the interval $(-\sqrt{n /\alpha}, \sqrt{ n /\alpha})$. It is readily seen that at each {\it node} $\pm X_k$
\begin{equation}\label{U-residue}
	\Res_{z=\pm X_k}
	U(z)
	=\lim_{z\to \pm X_k}
	U(z)
	\begin{pmatrix}
	0  &  w(z) \\
	0  &  0
\end{pmatrix},
\end{equation}
where
\begin{equation}\label{w}
	w(z):= n^{1/2} w_d(n^{-1/2} z),
\end{equation}
and as $z\to \infty$,
  \begin{equation}
  U(z)=\left (I+O\left (1 /z\right )\right )\left(
                               \begin{array}{cc}
                                 z^n & 0 \\
                                 0 & z^{-n} \\
                               \end{array}
                             \right).
\end{equation}

Define
\begin{equation}\label{Pi}
	\Pi(z):=
	{  \sin  \theta(z) }/
	{ \gamma(z) },	
\end{equation}
where
\begin{equation}
	\theta(z):=  n\pi/z^2 - \pi\alpha
\qquad	
\mbox{and}
\qquad
	\gamma(z):=-2n\pi/z^3.
\end{equation}
Note that for each $X_k\in X$
\begin{equation}\label{Pi-prop}
	\Pi(\pm X_k)=0
	\qquad \mbox{and}
	\qquad
	 \dfrac{[\sin\theta(
	 \pm X_k)]'}{\gamma(
	 \pm X_k)}=(-1)^k.
\end{equation}
Moreover, we introduce the upper triangular matrices
\begin{equation}\label{u}
	\mathcal{U}_\pm(z):=
	\begin{pmatrix}
	1  &   -\frac{w(z)}
	{\Pi (z)   }  e^{ \pm i \theta(z)}
	\\
	  0   &  1
\end{pmatrix},
\end{equation}
and the lower triangular matrices
\begin{equation}\label{u-tria}
	\mathcal{U}_\pm^{\triangle}(z):=
\Pi(z)^{-\sigma_3}
\begin{pmatrix}
	1  &  0
	\\
	   -\frac{1 }
	{\Pi (z) w(z) }  e^{ \pm i\theta(z)  }
   &  1
\end{pmatrix}
.
\end{equation}

\begin{figure}
\centering
\begin{tikzpicture}
\def\rad{1.8cm}
  \path [thick,draw=black,postaction={on each segment={mid arrow=black}}]
  (-3*\rad,0)--(-1.7*\rad,0)
  (-1.7*\rad,0)--(0,0)
  (0,0)--(1.7*\rad,0)
  (1.7*\rad,0)--(3*\rad,0)

  (-3*\rad,0)--(-3*\rad,\rad) node [above right] {$\quad \Sigma_{l,+}^\triangle$}
  (-3*\rad,\rad)--(-1.7*\rad,\rad)
  (-1.7*\rad,\rad)--(-\rad,\rad)
  (-1.7*\rad,\rad) --(-1.7*\rad,0)
  (-\rad,\rad)node [above] {$\Sigma_{l,+}$} --(0,0)

  (-3*\rad,0)--(-3*\rad,-\rad) node [below right] {$\quad \Sigma_{l,-}^\triangle$}
  (-3*\rad,-\rad)--(-1.7*\rad,-\rad)
  (-1.7*\rad,-\rad)--(-\rad,-\rad)
  (-1.7*\rad,-\rad)--(-1.7*\rad,0)
  (-\rad,-\rad)node [below] {$\Sigma_{l,-}$} --(0,0)

  (3*\rad,\rad)node [above left] {$\Sigma_{r,+}^\triangle\qquad$} --(3*\rad,0)
  (1.7*\rad,\rad)--(3*\rad,\rad)
  (\rad,\rad)--(1.7*\rad,\rad)
  (1.7*\rad,\rad)--(1.7*\rad,0)
  (0,0)-- (\rad,\rad) node [above] {$\Sigma_{r,+}$}

   (3*\rad,-\rad)node [below left] {$\Sigma_{r,-}^\triangle\qquad$} --(3*\rad,0)
  (1.7*\rad,-\rad)--(3*\rad,-\rad)
  (\rad,-\rad)--(1.7*\rad,-\rad)
  (1.7*\rad,-\rad)--(1.7*\rad,0)
  (0,0)--(\rad,-\rad) node [below] {$\Sigma_{r,-}$}

  ;
 \filldraw (-3*\rad,0) circle (1pt)
 node [left] {$-k_n$};

\filldraw (-1.7*\rad,0) circle (1pt)
  node [below right] {$-2$};

\filldraw (1.7*\rad,0) circle (1pt)
  node [below left] {$2$};

\filldraw (0,0) circle (1pt)
 node [below] {$0$};

 \filldraw (3*\rad,0) circle (1pt)
 node [right] {$k_n$};

\node [above] at (0,1.2) {$\Omega_\infty$};

\node [above] at (-2.3*\rad,.6) {$\Omega_{l,+}^\triangle$};

\node [above] at (2.3*\rad,.6) {$\Omega_{r,+}^\triangle$};

\node [below] at (-2.3*\rad,-.6) {$\Omega_{l,-}^\triangle$};

\node [below] at (2.3*\rad,-.6) {$\Omega_{r,-}^\triangle$};

\node [above] at (-1.1*\rad,.6) {$\Omega_{l,+}$};

\node [above] at (1.1*\rad,.6) {$\Omega_{r,+}$};

\node [below] at (-1.1*\rad,-.6) {$\Omega_{l,-}$};

\node [below] at (1.1*\rad,-.6) {$\Omega_{r,-}$};

\node [left] at (-3*\rad,\rad) {$-k_n+i\delta$};

\node [left] at (-3*\rad,-\rad) {$-k_n-i\delta$};

\end{tikzpicture}

\caption{The contour $\Sigma$ and the region $\Omega$.}
\label{figure1}
\end{figure}
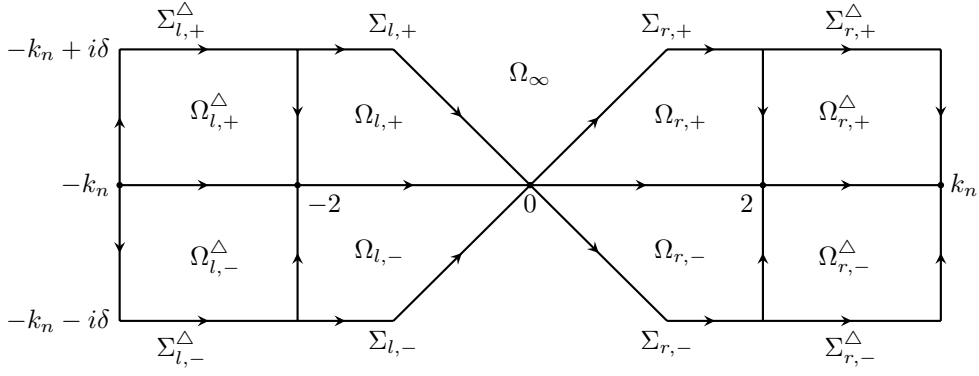

Let $\delta$ be a small positive number and $k_n=\sqrt{n/\alpha}+\delta$, 
we divide the complex plane into nine parts: $\Omega=\Omega_{l,\pm}\cup\Omega_{r,\pm}
\cup
\Omega_{l,\pm}^\triangle\cup\Omega_{r,\pm}^\triangle\cup \Omega_\infty
$ by the contour $\Sigma=(-k_n,k_n) \cup
\Sigma_{l,\pm}\cup\Sigma_{r,\pm}\cup \Sigma_{l,\pm}^\triangle  \cup \Sigma_{r,\pm}^\triangle$; see Figure \ref{figure1}.
Note that here we choose  the contour
like $\Sigma_{l,\pm}$ and $\Sigma_{r,\pm}$ so as to avoid the zeros of $\Pi(z)$ and $w(z)$; see \eqref{u} and \eqref{u-tria}.
To reduce the IP to a RHP, we shall employ an idea of Bleher and  Liechty  \cite{BleherE,Bleher} with some modifications.
Define
\begin{equation}\label{R}
	R(z):=
	U(z)\times
	\begin{cases}
	\vspace{.1cm}
	\mathcal{U}_+(z),  &  z\in \Omega_{l,+}\cup \Omega_{r,-},\\
	\vspace{.1cm}
	\mathcal{U}_-(z),  &  z\in \Omega_{l,-}\cup \Omega_{r,+},\\
	\vspace{.1cm}
	\mathcal{U}_+^\triangle(z),  &  z\in \Omega_{l,+}^\triangle\cup \Omega_{r,-}^\triangle,\\
	\vspace{.1cm}
	\mathcal{U}_-^\triangle(z),  &  z\in \Omega_{l,-}^\triangle\cup \Omega_{r,+}^\triangle,
	\\
	\vspace{.1cm}
	I,   &  z\in \Omega_\infty.
\end{cases}
\end{equation}

\begin{prop}
	For each $\pm X_k\in X$, the singularity of $R(z)$ at $\pm X_k$ is removable; that is, $\Res\limits_{z=\pm X_k}
R(z)=0$.
\end{prop}
\begin{proof}
    For $z\in \Omega_{r,\pm}$, it follows from \eqref{R} that
\begin{equation}\label{R-U-1}
	R_{11}(z)=U_{11}(z),
	\qquad
	R_{12}(z)
	=U_{12}(z)-U_{11}(z)
	\frac{w(z)}{\Pi(z)}e^{\pm i\theta(z)}.
\end{equation}
Consider any $X_k\in (0,2)$.
By \eqref{Pi-prop} and \eqref{R-U-1}, the residue of
  $R_{12}(z)$ at $X_k$ is given by
\begin{equation}
	\Res_{z=X_k} R_{12}(z)
	= w(X_k)U_{11}(X_k)-
	U_{11}(X_k) \frac{w(X_k)}{(-1)^k}
	(-1)^k
	=0
	.
\end{equation}
Similarly, we also have
$
	\Res\limits_{z=X_k}
	R_{22}(z)=0,
$
thus
\begin{equation}
	\Res_{z=X_k}
	R(z)=0 \quad
	\mbox{for }
	X_k\in (0,2).
\end{equation}

On the other hand,
it is readily seen that
for $z\in \Omega_{r,\pm }^\triangle$
\begin{equation}\label{R-U-2}
	R_{11}(z)=U_{11}(z)
	\Pi(z)^{-1}
	-w(z)^{-1}  U_{12}(z)
	e^{\pm i\theta(z)}
	,
	\qquad
	R_{12}(z)
	=U_{12}(z)\Pi(z).
\end{equation}
For $X_k\in (2,k_n)$,
since  the pole
of the entry $R_{12}(z)$
at $X_k$ is canceled by the zero
of the function $\Pi(z)$, $R_{12}(z)$ has no pole at $X_k$.
Moreover, from \eqref{Pi-prop} and \eqref{R-U-2} we have
\begin{equation}
	\Res_{z=X_k} R_{11}(z)
	= (-1)^k U_{11}(X_k)
	-
	\frac{1}{w(X_k)}
	U_{11}(X_k) w(X_k)
	(-1)^k
	=0.
\end{equation}
Similarly, it can be shown that
$
	\Res\limits_{z=X_k}
	R_{22}(z)=0,
$
thus,
\begin{equation}
	\Res_{z=X_k}
	R(z)=0, \qquad
	X_k\in (2,k_n).
\end{equation}

In the same way, we obtain that for
each $-X_k\in X$
\begin{equation}
	\Res_{z=-X_k}
	R(z)=0,
\end{equation}
hence, $R(z)$ has no pole at $\pm X_k\in X$.

\end{proof}

Note that this transformation makes $R_+(z)$ and $R_-(z)$ continuous on the interval $(-k_n,k_n)$.
As a consequence, we have created several jump discontinuities on the contour $\Sigma$ in the complex plane.
It is easily verified that $R(z)$ is a solution of the following RHP:

\begin{description}
  \item[($R_1$)]     $R(z)$ is analytic for
  $z\in\mathbb{C}\setminus \Sigma$;

  \item[($R_2$)]
  for $z\in\Sigma$, $R(z)$ satisfies
\[
    R_+(z)=R_-(z) J_R(z),
\]
where
for $z$ on the real line
\begin{equation}
	J_R(z)=
	\begin{cases}
\vspace{.1cm}
      \begin{pmatrix}
	1   &    0
		\\
	 -2i \gamma(z)/w(z)
	  & 1
	\end{pmatrix},  &  z\in(-k_n,-2),
\\
\vspace{.1cm}
		\begin{pmatrix}
		1   &    -2i \gamma(z)  w(z)  \\
			0  & 1
		\end{pmatrix},
		&
      z\in (-2,0),
\\
		\begin{pmatrix}
		1   &    2i \gamma(z)  w(z)  \\
			0  & 1
		\end{pmatrix},
		&
      z\in (0,2),
\\
      \begin{pmatrix}
	1   &    0
		\\
	 2i \gamma(z)/w(z)
	  & 1
	\end{pmatrix},  &  z\in(2,k_n),
\end{cases}
\end{equation}
for $z\in \Sigma_{l,\pm}\cup \Sigma_{l,\pm}^\triangle$
\begin{equation}
	J_R(z)=
	\begin{cases}
\vspace{.1cm}
\mathcal{U}_+(z)^{-1},
	&      z\in \Sigma_{l,+},
\\
\vspace{.1cm}
		\mathcal{U}_-(z),
		&      z\in \Sigma_{l,-},
\\
\vspace{.1cm}
\mathcal{U}_+^\triangle(z)^{-1},
	&      z\in \Sigma^\triangle_{l,+},
\\
\vspace{.1cm}
\mathcal{U}_-^\triangle(z),
	&      z\in \Sigma^\triangle_{l,-},
\end{cases}
\end{equation}
for $z\in \Sigma_{r,\pm}\cup \Sigma_{r,\pm}^\triangle$
\begin{equation}
	J_R(z)=
	\begin{cases}
\vspace{.1cm}
\mathcal{U}_-(z)^{-1},
	&      z\in \Sigma_{r,+},
\\
\vspace{.1cm}
		\mathcal{U}_+(z),
		&      z\in \Sigma_{r,-},
\\
\vspace{.1cm}
\mathcal{U}_-^\triangle(z)^{-1},
	&      z\in \Sigma^\triangle_{r,+},
\\
\vspace{.1cm}
\mathcal{U}_+^\triangle(z),
	&      z\in \Sigma^\triangle_{r,-},
\end{cases}
\end{equation}
and
\begin{equation}
	J_R(z)=
	\begin{cases}
\vspace{.1cm}
	\begin{pmatrix}
	\Pi(z)
	&   -w(z) e^{i\theta(z)}
	\\
	e^{i\theta(z)}/w(z)
	  &    -2i \gamma(z) e^{i\theta(z)}
	\end{pmatrix},
	&
             z\in -2 +(0,\delta)i \cup
             2- (0,\delta)i,
\\
	\begin{pmatrix}
	 2i \gamma(z)  e^{- i \theta(z)}
	&   w(z) e^{ - i\theta(z) }
	\\
	- e^{ - i \theta(z) } /w(z)
	  &   \Pi(z)
	\end{pmatrix},
           &   z\in -2- (0,\delta)i
           \cup 2+(0,\delta)i
           ;
\end{cases}
\end{equation}

\item[($R_3$)]    as $z\to \infty$,
  \begin{equation}\label{R infinity}R(z)=\left (I+O\left (1/z\right )\right )\left(
                               \begin{array}{cc}
                                 z^{n} & 0 \\
                                 0 & z^{-n} \\
                               \end{array}
                             \right);
\end{equation}
\item[($R_4$)] $R(z)$ has the following behavior as $z\to 0$,
     \begin{equation}
R(z)=O\left(
                               \begin{array}{cc}
                                 1 & \log|z| \\
                                 1 & \log|z| \\
                               \end{array}
                             \right).
                             \end{equation}

\end{description}

\section{The transformations $R\to T \to Q$}

To normalize the behavior at infinity,  we introduce the first transformation
\begin{equation}\label{T}
	T(z):=
	e^{ (-nl/2)\sigma_3 }
	R(z)  e^{ (-n g(z)+nl/2)
	\sigma_3 }.
\end{equation}

Now, we need some properties of the function $g(z)$.
In a similar manner as in \cite{LinWong},
the derivative of $g(z)$ can be calculated explicitly to give
\begin{equation}\label{g-prime}
 g'(z)= \frac{4}{z^3} \log( z+\sqrt{z^2-4} )
 +\frac{ \sqrt{z^2-4} }{ z^2 }  -\frac{4}{z^3}
 \log 2 \mp \frac{2 \pi}{  z^3} i
 \qquad \mbox{for } z\in\mathbb{C}_\pm.
\end{equation}

\begin{prop}

 The function $g'(z)$ satisfies
 \begin{equation}\label{g-prime-jump}
  g'_{\pm }(x)  =  \mp \pi i \psi(x),
  \qquad  x\in(-2,2);
 \end{equation}cf. \eqref{psi},
where $g_+'(x)$ and $g_-'(x)$ refer to the limiting values from the upper and lower half
planes, respectively.
Also,
it follows from \eqref{g-prime} and \eqref{g-prime-jump} that
\begin{equation}
    g_+(x)+g_-(x)
    \begin{cases}
      =l,  &   x\in (-2,2),
      \\
      >l,   &   x\in (-\infty,-2) \cup (2,+\infty);
    \end{cases}
\end{equation}where   $l=2\int_{-\infty}^{\infty} \log|2-s| \psi(s) ds$  is  the Lagrange multiplier;  cf. \eqref{l}.
Moreover, $g(z)$ is analytic in $\mathbb{C}\setminus \mathbb{R}$, and
\begin{equation}
 g(z) =\log z + O(1/z)
 \qquad \mbox{as } z\to\infty.
\end{equation}

\end{prop}

It is readily seen that $T(z)$ solves the following RHP:
\begin{description}
  \item[($T_1$)]     $Y(z)$ is analytic in
  $\mathbb{C}\setminus \Sigma$;

\item[($Y_2$)]    for $z\in\Sigma$
\begin{equation}
   T_+(z)=T_-(z)J_T(z),
\end{equation}
where
\begin{equation}
	J_T(z)
	=
	\begin{cases}
	e^{n (g_-(z)-l/2)\sigma_3 }J_R(z) e^{ -n(g_+(z) -l/2)\sigma_3 } & \mbox{for } z\in \mathbb{R},
	\\
	e^{n (g(z)-l/2)\sigma_3 }J_R(z) e^{ -n(g(z) -l/2)\sigma_3 } & \mbox{for } z\in \Sigma\setminus \mathbb{R};
\end{cases}
\end{equation}

\item[($T_3$)]    as $z \to \infty$,
 \begin{equation}\label{T_3}
  T(z)=I+O\left (1 /z\right );
\end{equation}
\item[($T_4$)]
$T(z)$ has the behavior
as $z\to 0$,
 \begin{equation}\label{T_4}
T(z)=O
                               \begin{pmatrix}
                                 1 & \log |z| \\
                                 1 & \log |z| \\
                               \end{pmatrix}.
                               \end{equation}
\end{description}
In particular, for $z\in \mathbb{R}$ we have
\begin{equation}\label{J_T}
	J_T(z)=
	\begin{cases}
\vspace{.1cm}
	\begin{pmatrix}
	  e^{2n\pi i /z^2}  &  0
	  \\
	     0  &   e^{-2n\pi i /z^2}
\end{pmatrix}, &   z\in (-\infty,-k_n),
\\
\vspace{.1cm}
\begin{pmatrix}
	e^{2i (\theta(z)
	 +\pi \alpha
	)}  &  0 \\
	-\frac{2i \gamma(z)}
	{w(z)}
	e^{ n(l-g_+(z)-g_-(z)) }
	   &  e^{-2i (\theta(z)
	 +\pi \alpha
	)}
\end{pmatrix},
   &  z\in(-k_n,-2),
\\
\vspace{.1cm}
	\begin{pmatrix}
	e^{-2n\phi_-(z)}  &  -2i\gamma(z)
	  w(z) 
	\\
	0  &  e^{-2n\phi_+(z)}
\end{pmatrix},
&  z\in(-2,0),
\\
\vspace{.1cm}
	\begin{pmatrix}
	e^{-2n\phi_-(z)}  &  2i\gamma(z)
	  w(z)  
	\\
	0  &  e^{-2n\phi_+(z)}
\end{pmatrix},
&  z\in(0,2),
\\
\vspace{.1cm}
   \begin{pmatrix}
	e^{-2i (\theta(z)
	 +\pi \alpha
	)}  &  0 \\
	\frac{2i \gamma(z)}
	{w(z)}
	e^{ n(l-g_+(z)-g_-(z)) }
	   &  e^{2i (\theta(z)
	 +\pi \alpha
	)}
\end{pmatrix},
   &  z\in(2,k_n),
\\
\vspace{.1cm}
	\begin{pmatrix}
	  e^{-2n\pi i /z^2}  &  0
	  \\
	     0  &   e^{2n\pi i /z^2}
\end{pmatrix}, &   z\in (k_n,+\infty).
\end{cases}
\end{equation}

Note that for the jump matrix on $(0,2)$, we have the following factorization
\begin{equation}\label{T-fact}
	\begin{pmatrix}
	e^{-2n\phi_-(z)}  &  2i\gamma(z)
	  w(z)
	\\
	0  &  e^{-2n\phi_+(z)}
\end{pmatrix}
=
  	\begin{pmatrix}
	1  &  0\\
	\frac{e^{2n\phi_-(z)}}{2i\gamma(z)w(z)} & 	1
\end{pmatrix}
  \begin{pmatrix}
	0	&	2i\gamma(z)w(z)\\
	-\frac{1}{2i\gamma(z)w(z)}	&	0
\end{pmatrix}
\begin{pmatrix}
	1  &  0\\
	\frac{e^{2n\phi_+(z)}}{2i\gamma(z)w(z)} & 	1
\end{pmatrix},
\end{equation}
which allows us to reduce the jump matrix $J_T(z)$ to a simpler one  on the line segment $(0,2)$.
Also, note that for $z\in\Sigma^\triangle_{r,+}$, the (1,1) entry of $J_T(z)$ is $\frac{ 1}{2i \gamma(z)}e^{ i \theta(z)}
(1-e^{- 2i\theta(z)})$, and it behaves as $\frac{ 1}{2i \gamma(z)}e^{ i \theta(z)}$  for large $n$.
Now we are in a position   to introduce the second transformation $T\to Q$:
\begin{equation}\label{Q}
	Q(z):=
	T(z)\times
	\begin{cases}
\vspace{.13cm}
	-\mathcal{V}_-^\triangle (z)
	  , &  z\in \Omega_{l,+}^\triangle
	  \cup\Omega_{r,-}^\triangle,
\\
\vspace{.13cm}	
	\mathcal{V}_+^\triangle (z)
	  , &  z\in
	  \Omega_{l,-}^\triangle \cup
	  \Omega_{r,+}^\triangle,
\\
\vspace{.13cm}	
\mathcal{V}_+ (z)
	  , &  z\in \Omega_{l,+}
	  \cup\Omega_{r,-},
\\
\vspace{.13cm}	
\mathcal{V}_- (z)
	  , &  z\in
	  \Omega_{l,-} \cup
	  \Omega_{r,+},
\\
\vspace{.13cm}	
     I,
     &  z\in \Omega_\infty,
\end{cases}
\end{equation}
where
\begin{equation}\label{V-pm}
 \mathcal{V}_\pm(z):=
 \begin{pmatrix}
	1  & 0 \\
	\pm \frac{ e^{2n\phi(z)} }
	{2i \gamma(z)w(z)}
	&  1
\end{pmatrix}
\qquad \mbox{and} \qquad
\mathcal{V}_\pm^\triangle(z):=
\begin{pmatrix}
	\frac{1}
	{2 i \gamma(z) }  e^{\pm i\theta(z) }   &  0  \\
	0  &  2i\gamma(z)
	  e^{\mp  i\theta(z) }
\end{pmatrix}.
\end{equation}
Note that for $z\in \Sigma\setminus (-2,2)$, we expect that
$$
	J_Q(z) \sim I.
$$
However, for $z\in\Sigma_{r,\pm}^\triangle$,
since $1-e^{\mp 2i \theta(z)} =O( 1)$
when $z= O(k_n)$, the jump matrix $J_Q(z) \not \sim I$ near the critical point $z=k_n$.
 The reason is that
the interval $(2,k_n)$ is a saturated region,
so special attention must be paid to the edges of the saturated regions (c.f. the contour $\Sigma_{l,\pm}^\triangle$ or $\Sigma_{r,\pm}^\triangle$); see \cite{Baik}
and \cite{WangWong}.

Next,  we introduce some auxiliary functions, known as
 the $D$-functions, which are analogous to the function used in \cite{LinWong} to remove the jumps of RHP near  the edges of the saturated regions.
Recall
\begin{equation}\label{D-recall}
D(z):= \frac{\Gamma(\alpha-n/z^2)}{e^{n/z^2}\sqrt{2\pi}}\left(-\frac{n}{z^2}\right)^{n/z^2-\alpha+1/2}\qquad
\mbox{for } z\in \mathbb{C}\setminus \mathbb{R},
\end{equation}introduced in \eqref{D},
and  we let
\begin{equation} \label{D-tilde}
	\widetilde D(z):=
	\frac{\sqrt{2\pi}}
	{\Gamma(1+n/z^2-\alpha)
	e^{n/z^2}
	 }
	 \left(
	 \frac{n}{z^2}
	 \right)^{
	 n/z^2 -\alpha +1/2
	 },
\end{equation}
{
where the branches are chosen as $\arg z\in (-\pi, \pi)$ in \eqref{D-tilde}, while in \eqref{D-recall},  $-1/z^2=e^{\pm \pi i}/z^2$ for $z\in \mathbb{C}_\pm$.}
It is readily seen that  the function $\widetilde D(z)$ is a non-zero
analytic function on $\mathbb{C} \setminus
	 (-\infty,0]$ such that
\begin{equation}\label{D-e-1}
	\widetilde D(z)=D(z)(1-e^{\mp 2i \theta(z)}),
	\qquad
	z\in \mathbb{C}_\pm.
\end{equation}
Moreover, we have as $n\to\infty$,
\begin{equation}\label{D-tilde-asy}
   \widetilde D(z)= 1+O(1/n),
\end{equation}
holding uniformly  for $z$ in any compact set $K\subset \mathbb{C}\setminus (-\infty,0]$, and
as $z\to\infty$,
\begin{equation}
   D(z)= \frac{  \Gamma(\alpha) }
   {\sqrt{2 \pi}} n^{1/2-\alpha}(-z^2)^{\alpha-1/2}  + O(1/z).
\end{equation}
The main difference between the function $D(z)$ and that in \cite{LinWong},
lies in the behavior as $z$ goes to infinity, therefore another function is needed in order to normalize the behavior at infinity.

Define
\begin{equation}\label{E}
	E(z):=
	\frac{
	\sqrt{2\pi}
	 }
	{ \Gamma(\alpha) }
	(4-z^2)^{1/2-\alpha},
	\qquad z\in \mathbb{C}\setminus
	(-\infty,-2) \cup (2,+\infty),
\end{equation}
where we always take the principal branches  of $(2-z)^{1/2-\alpha}$ and $(z+2)^{1/2-\alpha}$.
It is based on two observations. The first is that,
\begin{equation}
 D_+(z)E_+(z)
    =D_-(z)E_-(z)
    \times
\begin{cases}
    e^{2n\pi i /z^2},  &  z\in(2,+\infty), \\
    e^{-2n\pi i /z^2},  &  z\in(-\infty,-2),
\end{cases}
\end{equation}
which is exactly the jumps of $J_T(z)$ for $z\in(-\infty,-k_n)\cup (k_n,+\infty)$; see also
\eqref{J_T}, and
$D(z)E(z)$  tends to $n^{1/2-\alpha}$ as $z\to \infty$.
The second observation is the  fact that
\begin{equation}
[E_-(z) \widetilde D_-(z) ]^{-\sigma_3} J_T(z) [E_+(z)\widetilde D_+(z)]^{\sigma_3} \sim I,
\end{equation}
uniformly for $z\in (-k_n,-2) \cup (2,k_n)$.

Similarly, we set
\begin{equation} \label{D-hat}
	\widehat D(z):=
	\frac{\sqrt{2\pi}}
	{\Gamma(1+n/z^2-\alpha)
	e^{n/z^2}
	 }
	 \left(
	 \frac{n}{(- z)^2}
	 \right)^{
	 n/z^2 -\alpha +1/2
	 },
\end{equation}
which is analytic in
$z\in
	 \mathbb{C} \setminus
	 [0,+\infty)$, with the branch
  taken as $\arg(-z)\in (-\pi, \pi)$, and it satisfies
\begin{equation}\label{D-e-2}
	\widehat D(z)=D(z)(1-e^{\pm 2i \theta(z)}),
	\qquad
	z\in \mathbb{C}_\pm.
\end{equation}
Moreover,
\begin{equation}\label{E-tilde}
	\widetilde E(z):=
	\frac{
	\sqrt{2\pi}
	 }
	{ \Gamma(\alpha) }
	(z^2-4)^{1/2-\alpha},
	\qquad
	z\in \mathbb{C}\setminus (-\infty,2)
	,
\end{equation}
and
\begin{equation}\label{E-hat}
	\widehat E(z):=
	\frac{
	\sqrt{2\pi}
	 }
	{ \Gamma(\alpha) }
	(-z-2)^{1/2-\alpha}
	(2-z)^{1/2-\alpha}
	,
	\qquad
	z\in \mathbb{C}\setminus (-2,+\infty)
	.
\end{equation}

For simplicity, we also bring in the notations
\begin{equation}\label{w_0}
   w_0(z):= 2 i\gamma(z)
   w(z)/ E^2(z),
\end{equation}
and
\begin{equation}\label{w-tilde}
   \widetilde
    w(z)
    :=  -2i \gamma(z) w(z)/ \widetilde E^2(z),
\qquad
   \widehat
    w(z)
    :=  -2i \gamma(z) w(z)/ \widehat E^2(z).
\end{equation}

As we have mentioned before, the density function $\psi(x)$ attains    its upper
constraint at $x=\pm 2$, the so-called  band-saturated region  endpoints; see \eqref{psi}.
Furthermore, since $\psi(x)$ is not differentiable at the point $x=\pm 2$, the function $\phi(z)$ is not analytic there,  nor can we   construct our local parametrix there  (such as the Airy parametrix) by using $\phi(z)$. Therefore, 
for our future analysis, a few more auxiliary functions are needed. To this aim, we resume 
 the function  $\widetilde \phi(z)$ in \eqref{phi-tilde}, and  define
\begin{equation}\label{phi-hat}
	\widehat \phi(z):=
	\int_{-2}^z \left (
	-g'(s) \pm \frac{2\pi i}{s^3}
	\right)
	ds
	\qquad  \mbox{for } z\in \mathbb{C}_\pm,
\end{equation}
which is analytic in $\mathbb{C}
	\setminus [-2,+\infty)$.
The functions $\widetilde \phi(z)$ and $\widehat \phi(z)$ will  play an important role in our argument, and the following are some of their properties.

\begin{prop}\label{pro-phi-tilde}
	With $\phi(z)$ defined in \eqref{phi}, the following connection formulas between the $\phi-$function and the
$\widetilde \phi-$function ($\widehat \phi-$function) hold:
\begin{equation}\label{phi-tilde-phi}
	\widetilde \phi(z)=\phi(z)
	\pm \frac{\pi i}{z^2}
\end{equation}
and
\begin{equation}\label{phi-hat-phi}
	\widehat \phi(z)=\phi(z)
	\mp \left( \frac{1}{z^2}-1  \right)
	\pi i
\end{equation}
for $z\in \mathbb{C}_\pm$.

If $x\in(2,+\infty)$, we have
\begin{equation}\label{phi-tilde-0}
   \Im \widetilde \phi(x) =0,
   \quad \mbox{ and }
   \quad
   \Re \widetilde \phi(x) <0.
\end{equation}
Similarly, for $x\in(-\infty,-2)$,
\begin{equation}\label{phi-hat-0}
   \Im \widehat \phi(x) =0,
   \quad \mbox{ and }
   \quad
   \Re \widehat \phi(x) <0.
\end{equation}

   Moreover, for $x\in(-2,2)$, we have
\begin{equation}\label{phi-0}
   \Re \phi(x\pm i \varepsilon)  <0.
\end{equation}

\end{prop}
\begin{proof}
	For $z\in \mathbb{C}_+$, we have from \eqref{phi}
	and \eqref{phi-tilde}
	\begin{equation}
\begin{aligned}	
	\widetilde \phi(z)&=\int_2^z
	\left (  -g'(s) - \frac{2\pi i}
	{s^3}  \right)ds
=-g(z)+g_+(2)+ \frac{\pi i}{z^2}
-\frac{\pi i}{4}	
=\phi(z)+\frac{\pi i}{z^2},
\end{aligned}
\end{equation}
where in the last step we have made use of the result
$
g_\pm(2)=l/2\pm \pi i \int_2^{+\infty}
\psi(s)dx=l/2\pm {\pi i}/{4}
$.

Also, it follows that for $z\in\mathbb{C}_-$
\begin{equation}
\begin{aligned}	
	\widetilde \phi(z)
=-g(z)+g_-(2)- \frac{\pi i}{z^2}
+\frac{\pi i}{4}	
=\phi(z)-\frac{\pi i}{z^2}
,
\end{aligned}	
\end{equation}
thus proving \eqref{phi-tilde-phi}.
For $z\in\mathbb{C}_\pm$, \eqref{phi-hat-phi} can be proven in a similar way.

Since by \eqref{phi-tilde-phi}
$
\Im \widetilde \phi(z) = \Im\left\{ \phi(z)
\pm {\pi i}/{z^2}\right \}
$, from \eqref{phi}  we get $\Im \widetilde \phi(x)=0$.
Also, note that
for  $x\in(2,+\infty)$ the upper constraint on the density $\psi(x)$ implies
\begin{equation}
	\Re \widetilde \phi(x)=l/2-
	\int_{-\infty}^{+\infty}
	\log |x-s| \psi(s) ds<0.
	\qquad
\end{equation}
Similarly, we can prove \eqref{phi-hat-0}.

On account of \eqref{g-prime-jump},
we obtain for $x\in(-2,2)$
   \begin{equation}
	\Re \phi(x\pm i \varepsilon)
	=
	-\Re \int_2^x g_\pm'(s) ds
	-\Re \int_x^{x\pm i\varepsilon}
	g'(s)ds
	=-\Re \int_x^{x\pm i\varepsilon}
	g'(s)ds
	=-\pi \varepsilon \psi(x)
	+O(\varepsilon^2)<0.
\end{equation}
\end{proof}



We conclude this section with a calculation of $l$ given in \eqref{phi}. First we note that $\widetilde \phi(z)$
is analytic in $\mathbb{C}\setminus (-\infty,2]$.
Coupling \eqref{g-prime} with \eqref{phi-tilde} gives
\begin{equation}\label{phi-tilde-exact}
   \widetilde  \phi(z)
   =   \frac{2}{z^2} \log(z+\sqrt{z^2-4})
   - \log(z+\sqrt{z^2-4})+
   (1-\frac{2}{z^2}) \log 2
   +\frac{\sqrt{z^2-4}}{2 z}
\end{equation}
for $z\in\mathbb{C}\setminus(-\infty,2]$.
In view of \eqref{phi-tilde}, we have
\begin{equation}\label{t39}
 l/2 - g(z) - \widetilde \phi(z) \pm \frac{\pi i}{z^2}=0
 \qquad \mbox{for } z\in \mathbb{C}_\pm.
\end{equation}
Now let $z\to\infty$; on account of \eqref{g}, \eqref{phi-tilde-exact} and \eqref{t39}, we obtain
\begin{equation}
 l = \lim_{z\to\infty} 2(\log(z) + \widetilde \phi(z)) =1.
 \label{l}
\end{equation}

\section{The transformation $Q\to V$   and the parametrix for the outside region}
Using the functions introduced  in Section 4, we take the third transformation $Q\to V$
\begin{equation}\label{V}
	V(z):=
	n^{(\alpha-1/2)\sigma_3}
	Q(z)
	E(z)^{\sigma_3}
	\times
	\begin{cases}
	\vspace{.1cm}
	\widehat D(z) ^{\sigma_3},
	&  z\in \Omega_{l,\pm}\cup \Omega_{l,\pm}^\triangle,
	\\
	\vspace{.1cm}
	\widetilde D(z) ^{\sigma_3},
	&  z\in \Omega_{r,\pm}\cup \Omega_{r,\pm}^\triangle,
	\\
	\vspace{.1cm}
	 D(z)^{\sigma_3},
	&  z\in \Omega_{\infty};
\end{cases}
\end{equation}
see   \cite{LinWong} and \cite{WangWong} for similar transformations, and see  Figure \ref{figure1}  for the regions.

\begin{prop}

The matrix-valued function $V(z)$ has the following jumps on the contour $\Sigma$:
\begin{equation}
	V_+(z)=V_-(z) J_V(z),
\end{equation}	
where
\begin{equation}\label{V-jump-l}
	J_V(z)=
	\begin{cases}
\vspace{.1cm}	
	\begin{pmatrix}
	0  &   -{w_0(z)}/{
	\widehat D^2 (z)}
	\\
	\widehat D^2(z)/w_0(z)
	&  0
\end{pmatrix},  &  z\in(-2,0),
\\
\vspace{.1cm}
	\begin{pmatrix}
	1  &    0
	\\
	-\frac{\widehat D^2(z)}{\widehat w(z)}
	e^{2n\widehat \phi(z)}
	&  1
\end{pmatrix},  &  z\in  (-k_n,-2),
\\
\vspace{.1cm}
      \begin{pmatrix}
	\frac{D(z)}{\widehat D(z)}
	&   -\frac{\widehat w(z)}
	{\widehat D^2(z)} e^{-2n
	\widehat \phi(z)}
	\\
	-\frac{D(z) \widehat D(z)}
	{w_0(z)} e^{2n\phi(z)}
	&  1
\end{pmatrix},
    &  z\in \Sigma_{l,+},
\\
\vspace{.1cm}
      \begin{pmatrix}
	1
	&   -\frac{\widehat w(z)}
	{\widehat D^2(z)} e^{-2n
	\widehat \phi(z)}
	\\
	-\frac{D(z) \widehat D(z)}
	{w_0(z)} e^{2n\phi(z)}
	&  \frac{D(z)}{\widehat D(z)}
\end{pmatrix},
    &  z\in \Sigma_{l,-},
\\
\vspace{.1cm}
   \begin{pmatrix}
	1 &   \frac{\widehat w(z)}{  \widehat D^2(z) }
	  e^{-2n\widehat \phi(z)}
	\\
	0 &  1
\end{pmatrix},  &  z\in -2 \pm i(0,\delta),
\\
\vspace{.1cm}
	\begin{pmatrix}
	1  &  0 \\
	-\frac{D(z) \widehat D(z)}
	{w_0(z)}e^{2n \phi (z)}
	&  1
\end{pmatrix},  &  z\in \Sigma_{l,\pm}^\triangle,
\end{cases}
\end{equation}
and
\begin{equation}\label{V-jump-r}
	J_V(z)=
	\begin{cases}
\vspace{.1cm}	
	\begin{pmatrix}
	0  &   {w_0(z)}/{
	\widetilde D^2 (z)}
	\\
	-\widetilde D^2(z)/w_0(z)
	&  0
\end{pmatrix},  &  z\in(0,2),
\\
\vspace{.1cm}
	\begin{pmatrix}
	1  &    0
	\\
	\frac{\widetilde D^2(z)}{\widetilde w(z)}
	e^{2n\widetilde \phi(z)}
	&  1
\end{pmatrix},  &  x\in  (2,k_n),
\\
\vspace{.1cm}
      \begin{pmatrix}
	\frac{D(z)}{\widetilde D(z)}
	&   \frac{\widetilde w(z)}
	{\widetilde D^2(z)} e^{-2n
	\widetilde \phi(z)}
	\\
	\frac{D(z) \widetilde D(z)}
	{w_0(z)} e^{2n\phi(z)}
	&  1
\end{pmatrix},
    &  z\in \Sigma_{r,+},
\\
\vspace{.1cm}
      \begin{pmatrix}
	1
	&   \frac{\widetilde w(z)}
	{\widetilde D^2(z)} e^{-2n
	\widetilde \phi(z)}
	\\
	\frac{D(z) \widetilde D(z)}
	{w_0(z)} e^{2n\phi(z)}
	&  \frac{D(z)}{\widetilde D(z)}
\end{pmatrix},
    &  z\in \Sigma_{r,-},
\\
\vspace{.1cm}
   \begin{pmatrix}
	1 &   \frac{\widetilde w(z)}{  \widetilde D^2(z) }
	  e^{-2n\widetilde \phi(z)}
	\\
	0 &  1
\end{pmatrix},  &  z\in 2 \pm i(0,\delta),
\\
\vspace{.1cm}
	\begin{pmatrix}
	1  &  0 \\
	\frac{D(z) \widetilde D(z)}
	{w_0(z)}e^{2n \phi (z)}
	&  1
\end{pmatrix},  &  z\in \Sigma_{r,\pm}^\triangle.
\end{cases}
\end{equation}
\end{prop}

It is readily seen that $V(z)$ satisfies the following RHP:
\begin{description}
  \item[($V_1$)]     $V(z)$ is analytic for
  $z\in\mathbb{C}\setminus \Sigma$;

  \item[($V_2$)]
  for $z\in\Sigma$, $V(z)$ satisfies
\[
    V_+(z)=V_-(z) J_V(z);
\]
\item[($V_3$)]    as $z\to \infty$,
  \begin{equation}\label{V-inf}
  V(z)=I+O\left (1/z\right);
\end{equation}
\item[($V_4$)]
as $z\to 0$,
$V(z)$ has the behavior
     \begin{equation}
V(z)=O(\log|z|);
                             \end{equation}
\item[($V_5$)]
as $z\to \pm2$,
$V(z)$ has the behavior
     \begin{equation}
V(z)=O\left(1\right)(z^2-4)^{(1/2-\alpha)\sigma_3}.
                             \end{equation}

\end{description}

From Proposition \ref{pro-phi-tilde}, we  can choose $\delta$ sufficiently small so that
\begin{equation}
 \Re \phi(z) <0,
 \quad
 \Re \widetilde \phi(z)>0
 \quad
 \mbox{ and }
 \quad
 \Re \widehat \phi(z)>0
\end{equation}
in both the upper and lower lens regions. These together with \eqref{phi-hat-0}-\eqref{phi-0}, \eqref{D-e-1} and \eqref{D-e-2},
imply that all jumps on the contour $\Sigma$ are exponentially close to the identity matrix, provided that they are bounded away from the segment $(-2,2)$.
Moreover,  for $x\in(-2,2)$, the functions
 $-w_0(x)/\widehat D^2(x)$
and $w_0(x)/\widetilde D^2(x)$ are approximated by
$
	-(4-x^2)^{2\alpha -1}
	e^{L}
$
as $n\to \infty$, where $L=\log\left(
\sqrt 2 i \Gamma^2(\alpha) e^\alpha /\sqrt \pi
\right)$.
It is therefore natural to  expect  that for large $n$, the solution of the RHP for $V(z)$ may behave asymptotically like the solution of the following RHP for $N(z)$:
\begin{description}
  \item[($N_1$)] $N(z)$ is analytic in $\mathbb{C}\backslash [-2,2]$;
\item[($N_2$)]    for $x\in(-2,2)$,
\begin{equation}
\begin{array}{llll}
 N_+(x)=N_-(x) \left(
                               \begin{array}{cc}
                                 0& -(4-x^2)^{2\alpha-1}e^{L} \\
                                 (4-x^2)^{1-2\alpha}e^{-L} & 0 \\
                                 \end{array}
                             \right);
                            \\
\end{array}
\end{equation}
\item[($N_3$)]    as $z\to\infty$,
  \begin{equation}\label{N-inf}
  N(z)=I+O(1/z).
  \end{equation}
\end{description}

This problem can be solved explicitly, and its solution is given by
\begin{equation}
  N(z)=\left(\begin{array}{cc}
  (z^2-4)^{1/4-\alpha} \varphi(z/2)^{2\alpha-1/2} & -ie^L(z^2-4)^{\alpha-3/4} \varphi(z/2)^{1/2-2\alpha}\\
  ie^{-L}(z^2-4)^{1/4-\alpha}
  \varphi(z/2)^{2\alpha-3/2} &
  (z^2-4)^{\alpha-3/4}
  \varphi(z/2)^{3/2-2\alpha}
  \end{array}\right),
\label{N}
\end{equation}
where $\arg (z\pm 2)\in (-\pi, \pi)$, and 
\begin{equation}\label{varphi}
 \varphi(z):={z+\sqrt{z^2-1} }
\end{equation}
with a branch cut along $[-1, 1]$ and $\varphi(z)\to 2z$ as $z\to\infty$; see also \cite{Arno}.

\section{Local parametrices      and the final transformation $V\to S$}

In this section, we will consider local parametrices in small disks $U(0,\varepsilon)$,
$U(2,\varepsilon)$ and $U(-2,\varepsilon)$,  respectively,  centered at the origin and at the end points of bands which are adjacent to a saturated region.

\subsection{Parametrix at endpoints of the saturated-band region}

We seek a local parametrix $V_{loc}(z)$ defined on $U(2,\varepsilon)\cup U(-2,\varepsilon)$ such that
\begin{description}
  \item[($V_{loc,1}$)] $V_{loc}(z)$ is analytic in $U(2,\varepsilon)\cup U(-2,\varepsilon) \backslash \Sigma$;
\item[($V_{loc,2}$)]    for $z\in (U(2,\varepsilon)\cup U(-2,\varepsilon)) \cap \Sigma$,
\begin{equation}\label{V-2-jump}
 (V_{loc})_+(z)=(V_{loc})_-(z)
 J
 _V(z);
\end{equation}

\item[($V_{loc,3}$)]  for $z\in \partial U(2,\varepsilon)\cup \partial U(-2,\varepsilon)$,
  \begin{equation}
 V_{loc}(z)N^{-1}(z)=I+O\begin{pmatrix}
                              1/n
                                 \end{pmatrix}
 \qquad \mbox{as } n\to \infty;
  \end{equation}
  
 \item[($V_{loc,4}$)]
as $z\to \pm2$,
$V_{loc}(z)$ has the behavior
     \begin{equation}
V(z)=O\left(1\right)(z^2-4)^{(1/2-\alpha)\sigma_3}.
                             \end{equation} 
\end{description}

At first, we  construct the parametrix near the point $z=2$. The jumps $J_V(z)$ are given by
\begin{equation}
J_V(z)=
\left\{\begin{array}{llll}
  \left(
                               \begin{array}{cc}
                                 0&  w_0/\widetilde D^2 \\
                                 -\widetilde D^2/ w_0 & 0 \\
                                 \end{array}
                             \right), & z\in(2-\varepsilon,2),\\
 \left(
                              \begin{array}{cc}
                                 1& \widetilde w /(\widetilde D^2 e^{2n\widetilde \phi}) \\
                                 0 & 1 \\
                                 \end{array}
                             \right), & z\in(2,2+i\varepsilon),\\
 \left(
                                \begin{array}{cc}
                                 1& \widetilde w /(\widetilde D^2 e^{2n\widetilde \phi}) \\
                                 0 & 1 \\
                                 \end{array}
                             \right), & z\in(2,2-i\varepsilon),\\
 \left(
                               \begin{array}{cc}
                                 1& 0 \\
                                 \widetilde D^2 e^{2n\widetilde \phi}/\widetilde w & 1 \\
                                 \end{array}
                             \right), & z\in(2,2+\varepsilon).
\end{array}\right.
\end{equation}
If we set
\begin{equation}\label{V-2-tilde}
	P(z)
	=
	V_{loc}(z)
	\times
	\begin{cases}
	(\widetilde w^{1/2}/ \widetilde D)^{\sigma_3}
	e^{-n\widetilde \phi(z)\sigma_3}
	&
	\mbox{for } z\in U(2,\varepsilon),
	\\
	((-\widehat w)^{1/2}/
	\widehat D)^{\sigma_3}
	e^{-n\widehat \phi(z)\sigma_3}
	&
	\mbox{for } z\in U(-2,\varepsilon),
	\end{cases}
\end{equation}
then the jump conditions on $P(z)$ become
\begin{equation}
	P_+(z) =P_-(z)
	J_{P}(z),	
\end{equation}
where for $z\in U(2,\varepsilon)$
\begin{equation}
	J_{P}(z)
	=
	\begin{cases}
	\begin{pmatrix}
	0  &  -1  \\
	1  &  0
\end{pmatrix},
    	&  z\in(2-\varepsilon,2),
    	\\
    	\begin{pmatrix}
	1  &  1  \\
	0  &  1
\end{pmatrix},
    	&  z\in(2,2+i\varepsilon),
    	\\
    	   	\begin{pmatrix}
	1  &  1  \\
	0  &  1
\end{pmatrix},
    	&  z\in(2,2-i\varepsilon),
    	\\
    	   	\begin{pmatrix}
	1  &  0  \\
	1  &  1
\end{pmatrix},
    	&  z\in(2,2+\varepsilon);
\end{cases}
\end{equation}
see Figure \ref{figure-2}.

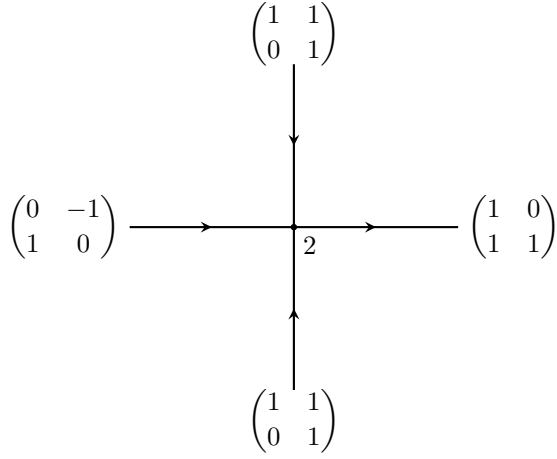
\begin{figure}
\centering
\begin{tikzpicture}
\def\rad{1.8cm}
  \path [thick,draw=black,postaction={on each segment={mid arrow=black}}]
  (0,1.2*\rad)--(0,0)
  (0,-1.2*\rad)--(0,0)
  (-1.2*\rad,0)--(0,0)
  (0,0)--(1.2*\rad,0)
  ;

\filldraw  (0,0) circle (1pt) node [below right] {$2$};

\node [above] at (0,1.1*\rad) {$
 \begin{pmatrix}
	1  & 1 \\
	0 & 1
\end{pmatrix}
$};

\node [left] at (-1.2*\rad,0) {$
 \begin{pmatrix}
	0  & -1 \\
	1 & 0
\end{pmatrix}
$}
 ;

\node [right] at (1.2*\rad,0) {$
 \begin{pmatrix}
	1  & 0 \\
	1 & 1
\end{pmatrix}
$}
 ;

 \node [below] at (0,-1.1*\rad) {$
 \begin{pmatrix}
	1  & 1 \\
	0 & 1
\end{pmatrix}
$}
 ;

\end{tikzpicture}

\caption{The contour near $z=2$.}
\label{figure-2}
\end{figure}

The jumps suggest that $P(z)$ can be constructed  by using Airy functions.  In view of the relations $\Ai(z)+\omega \Ai(\omega z)+\omega^2\Ai(\omega^2z)=0$ and $\Ai'(z)+\omega^2 \Ai'(\omega z)+\omega \Ai'(\omega^2z)=0$, we define
\begin{equation}
\Psi(z)
:=
\left\{
\begin{array}{llll}
\left(\begin{array}{cc}
\Ai(z) & \omega^2 \Ai(\omega^2z)\\
i\Ai'(z) & i\omega \Ai'(\omega^2z)
\end{array}\right) & \mbox{for} ~\arg z\in(0,\pi/2),\\
\\
\left(\begin{array}{cc}
-\omega \Ai(\omega z) & \omega^2 \Ai(\omega^2 z)\\
-i\omega^2 \Ai'(\omega z) & i\omega \Ai'(\omega^2 z)
\end{array}\right) & \mbox{for} ~\arg z\in(\pi/2,\pi),\\
\\
\left(\begin{array}{cc}
-\omega^2 \Ai(\omega^2z) & -\omega \Ai(\omega z)\\
-i\omega \Ai'(\omega^2z) & -i\omega^2 \Ai'(\omega z)
\end{array}\right) & \mbox{for} ~\arg z\in(-\pi,-\pi/2),\\
\\
\left(\begin{array}{cc}
\Ai(z) & -\omega \Ai(\omega z)\\
i \Ai'(z) & -i\omega^2 \Ai'(\omega z)
\end{array}\right) & \mbox{for} ~\arg z\in(-\pi/2,0),\\
\end{array}\right.
\label{Psi function}
\end{equation}
where $\omega=e^{2\pi i/3}$, and it satisfies the following jumps
\begin{equation}
\Psi_+(z)=\Psi_-(z)\left\{\begin{array}{llll}
 \left(
                               \begin{array}{cc}
                                 0& -1 \\
                                 1 & 0 \\
                                 \end{array}
                             \right) & \mbox{for} ~z\in(-\infty,0),\\
 \left(
                                \begin{array}{cc}
                                 1& 0 \\
                                 -1 & 1 \\
                                 \end{array}
                             \right) & \mbox{for} ~z\in(\infty e^{-\pi i/2},0)\cup (\infty e^{\pi i/2},0),\\
 \left(
                               \begin{array}{cc}
                                 1& -1 \\
                                 0 & 1 \\
                                 \end{array}
                             \right) & \mbox{for} ~z\in(0,\infty).
\end{array}\right.
\end{equation}

To construct our parametrix, we recall the function
\begin{equation}
	\widetilde f_n(z) =
	\left(
	  -\frac32 n \widetilde \phi (z)
	 \right)^{2/3};
\end{equation}
see \eqref{f-tilde}.
Since by \eqref{phi-tilde-exact}
\begin{equation}
  \widetilde \phi(z) =-\frac{2}{3}  (z-2)^{3/2}
  +O( (z-2)^2 ),
\end{equation}
it is readily seen that $\widetilde f_n(z)$ is a one-to-one mapping of $U(2,\varepsilon)$ onto a neighborhood of the origin.
For $ z\in U(2,\varepsilon)$,
comparing the RHP for $\Psi$ with the RHP for $P$ invokes us to seek a solution to the RHP for $V
_{loc}$ in the form of
\begin{equation}\label{V-2}
 V_{r}(z)=
  E_r(z)\sigma_2\Psi(\widetilde f_n(z))\sigma_2e^{n\widetilde\phi\sigma_3}
 (\widetilde D/
 \widetilde w^{1/2}
 )^{\sigma_3},
\end{equation}
where the Pauli matrix $\sigma_2=
\begin{pmatrix}
	0 & -i\\
	i  &  0
\end{pmatrix}
$
and the matrix-valued function $E_r(z)$ is analytic in $U(2,\varepsilon)$, to be  determined later.
Note that
\begin{equation}\label{par-1}
\sigma_2\Psi(\widetilde f_n(z))\sigma_2e^{n\widetilde\phi\sigma_3}=\frac 1{\sqrt{\pi}} (\widetilde f_n(z))^{\sigma_3/4}
\begin{pmatrix}
1 & -i\\
-i & 1
\end{pmatrix}
^{-1}
\left(I+O\left( 1/{n}\right)\right)
\end{equation}
and
\begin{equation}\label{par-2}
\begin{aligned}
{(\widetilde D/\widetilde w^{1/2})}^{\sigma_3}N(z){(\widetilde D/\widetilde w^{1/2})}^{-\sigma_3}
=\frac{(z^2-4)^{(1/2-\alpha)\sigma_3}}{(z^2-4)^{1/4}}
\left(\begin{array}{cc}
   \varphi(z/2)^{2\alpha-1/2} & -i\varphi(z/2)^{1/2-2\alpha}\\
  i\varphi(z/2)^{2\alpha-3/2} &
  \varphi(z/2)^{3/2-2\alpha}
   \end{array}\right)
   (I+O(1/n))
   ,
\end{aligned}
\end{equation}
which  holds uniformly  for  $z\in \partial U(2,\varepsilon)$ as $n\rightarrow\infty$.
Thus, the condition $V_{loc,3}$
with \eqref{par-1} and \eqref{par-2} leads us to set
\begin{equation}
E_r(z)=\sqrt{\pi}
(\widetilde D/\widetilde w^{1/2})^{-\sigma_3}
m_r(z) (\widetilde f_n(z))^{-\sigma_3/4},
\label{E-r}
\end{equation}
where
\begin{equation}
m_r(z)=\frac{(z^2-4)^{(1/2-\alpha)\sigma_3}}{(z^2-4)^{1/4}}\left(\begin{array}{cc}
     \varphi(z/2)^{2\alpha-1/2} & -i\varphi(z/2)^{1/2-2\alpha}\\
  i\varphi(z/2)^{2\alpha-3/2} &
  \varphi(z/2)^{3/2-2\alpha},
  \end{array}\right)
  \left(\begin{array}{cc}
  1 & -i\\
  -i & 1
  \end{array}\right).
\label{m-r}
\end{equation}
It can be shown that ${E}_r(z)$ has no jumps across the  interval $(0,2)$,
and in addition that it has a removable singularity at $z=2$.

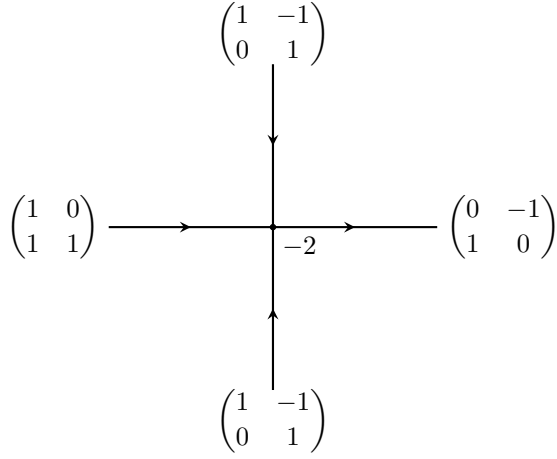
\begin{figure}
\centering
\begin{tikzpicture}
\def\rad{1.8cm}
  \path [thick,draw=black,postaction={on each segment={mid arrow=black}}]
  (0,1.2*\rad)--(0,0)
  (0,-1.2*\rad)--(0,0)
  (-1.2*\rad,0)--(0,0)
  (0,0)--(1.2*\rad,0)
  ;

\filldraw  (0,0) circle (1pt) node [below right] {$-2$};

\node [above] at (0,1.1*\rad) {$
 \begin{pmatrix}
	1  & -1 \\
	0 & 1
\end{pmatrix}
$};

\node [left] at (-1.2*\rad,0) {$
 \begin{pmatrix}
	1  & 0 \\
	1 & 1
\end{pmatrix}
$}
 ;

\node [right] at (1.2*\rad,0) {$
 \begin{pmatrix}
	0  & -1 \\
	1 & 0
\end{pmatrix}
$}
 ;

 \node [below] at (0,-1.1*\rad) {$
 \begin{pmatrix}
	1  & -1 \\
	0 & 1
\end{pmatrix}
$}
 ;

\end{tikzpicture}

\caption{The contour near $z=-2$.}
\label{figure--2}
\end{figure}

A similar construction gives the parametrix at the point $z=-2$. Namely, if
we set
\begin{equation}
	\widehat f_n(z) =
	\left(
	  -\frac32 n \widehat \phi (z)
	 \right)^{2/3},
\end{equation}
which is analytic in $U(-2,\varepsilon)$, and  is a one-to-one mapping of $U(-2,\varepsilon)$ onto the neighborhood of the origin. For $z\in U(-2,\varepsilon)$, the jumps $J_P(z)$  are given by
\begin{equation}
	J_{P}(z)
	=
	\begin{cases}
	\begin{pmatrix}
	0  &  -1  \\
	1  &  0
\end{pmatrix},
    	&  z\in(-2,-2+\varepsilon),
    	\\
    	\begin{pmatrix}
	1  &  -1  \\
	0  &  1
\end{pmatrix},
    	&  z\in(-2,-2+i\varepsilon),
    	\\
    	   	\begin{pmatrix}
	1  &  -1  \\
	0  &  1
\end{pmatrix},
    	&  z\in(-2,-2-i\varepsilon),
    	\\
    	   	\begin{pmatrix}
	1  &  0  \\
	1  &  1
\end{pmatrix},
    	&  z\in(-2-\varepsilon,-2);
\end{cases}
\end{equation}
see Figure \ref{figure--2}. These jump conditions are satisfied by the function
$
	\sigma_1\Psi(-z) \sigma_1
$,
where the pauli matrix $\sigma_1=
 \begin{pmatrix}
	0  &  1\\
	1  &  0
\end{pmatrix}
$. Hence, we can take for $z\in U(-2,\varepsilon)$
\begin{equation}\label{V-l}
	 V_{l}(z)=
  E_l(z)\sigma_1\Psi(\widehat f_n(z))\sigma_1
  e^{n\widehat \phi\sigma_3}
 (\widehat D/
 (-\widehat w)^{1/2}
 )^{\sigma_3},
\end{equation}
where
\begin{equation}
E_l(z)=\sqrt{\pi}
(\widehat D/(-\widehat w)^{1/2})^{-\sigma_3}
m_l(z) (\widehat f_n(z))^{-\sigma_3/4}
\label{E-l}
\end{equation}
and
\begin{equation}
{m}_l(z)=\frac{[(-2-z)(2-z)]^{(1/2-\alpha)\sigma_3}}{(2-z)^{1/4}(-2-z)^{1/4}}\left(\begin{array}{cc}
    \varphi(-z/2)^{2\alpha-1/2} & i \varphi(-z/2)^{1/2-2\alpha}\\
  -i \varphi(-z/2)^{2\alpha-3/2} &
  \varphi(-z/2)^{3/2-2\alpha}
  \end{array}\right)
  \left(\begin{array}{cc}
  1 & i\\
  i & 1
  \end{array}\right).
\label{m-l}
\end{equation}

\subsection{Parametrix at the point $z=0$}

As $z\to 0$ for $z\in \Sigma$, $J_V(z)-I$ is no longer small.  Hence we are led to the following parametrix for $V_0(z)$ at the origin:
\begin{description}
\item[($V_{0,1}$)] $V_0(z)$ is analytic in $U(0,\varepsilon)\backslash \Sigma$;
\item[($V_{0,2}$)]  for $z\in U(0,\varepsilon)\cap \Sigma$,  $V_0(z)$ shares the same jumps as $V(z)$; cf. \eqref{V-jump-l} and \eqref{V-jump-r},  see Figure \ref{figure-0-1}
for an illustration of the  contours;
\item[($V_{0,3}$)]  for $z\in \partial U(0,\varepsilon)$, the matching condition holds:
  \begin{equation}\label{V-0-matching}
 V_0(z)N^{-1}(z)=I+O\(1/n\)
 \qquad \mbox{as } n\to\infty.
  \end{equation}
\end{description}

Next, we proceed to find an approximating RHP, of which the solution $V_{loc}^0(z)$ is the leading order approximation of $V_0(z)$. Indeed, on account of \eqref{phi-tilde-exact} and \eqref{phi-tilde-phi},   we observe that
$e^{-n \widetilde \phi}$ is exponentially small and $D/\widetilde D$ tends to 1 exponentially and uniformly for $z\in \Sigma_{r,\pm}$, so long as $\arg z$ keeps a distance to $\pm \pi/2$. 
By simplifying the jump conditions for the RHP for $V$, we wish to find a matrix-valued function
$V_{loc}^0(z)$ such that
\begin{description}
  \item[($V_{loc,1}^0$)] $V_{loc}^0(z)$ is analytic in $U(0,\varepsilon)\backslash \Sigma$;
\item[($V_{loc,2}^0$)]    for $z\in U(0,\varepsilon)\cap \Sigma$,
\begin{equation}
 (V_{loc}^0)_+(z)=(V_{loc}^0)_-(z)J_{V_{loc}^0}(z),
\end{equation}
where
\begin{equation}
J_{V_{loc}^0}(z)=
\left\{\begin{array}{llllll}
 \left(
                              \begin{array}{cc}
                                 1& 0 \\
                                 -(4-z^2)^{1-2\alpha}e^{-L+2n\phi} & 1 \\
                                 \end{array}
                             \right) & \mbox{for} ~z\in
                           U(0,\varepsilon)\cap(    \Sigma_{l,\pm}\cup\Sigma_{r,\pm}),\\
  \left(
                               \begin{array}{cc}
                                 0 & -(4-z^2)^{2\alpha-1}e^{L} \\
                                 (4-z^2)^{1-2\alpha}e^{-L} & 0 \\
                                 \end{array}
                             \right) & \mbox{for} ~z\in(-\varepsilon,\varepsilon),
\end{array}\right.
\end{equation}with $\arg  (2-z)\in (-\pi, \pi)$ and $\arg (z+2)\in (-\pi, \pi)$;
\item[($V_{loc,3}^0$)]  for $z\in \partial U(0,\varepsilon)$,
  \begin{equation}\label{}
 V_{loc}^0(z)N^{-1}(z)=I+O\(1/n\)
 \qquad \mbox{as } n\to\infty.
  \end{equation}
\end{description}

\begin{figure}
\centering
\begin{tikzpicture}
\def\rad{1.8cm}
  \path [thick,draw=black,postaction={on each segment={mid arrow=black}}]
  (140:1.7*\rad) node [left] {$\Sigma_{l,+}$}
  --(0,0)
  (0,0)--(-40:1.7*\rad) node [right] {$\Sigma_{r,-}$}
  (-140:1.7*\rad)
  node [left] {$\Sigma_{l,-}$}
  --(0,0)
  (0,0)--(40:1.7*\rad)
  node [right] {$\Sigma_{r,+}$}
  (-1.6*\rad,0)--(0,0)
  (0,0)--(1.6*\rad,0)
  ;

\filldraw  (0,0) circle (1pt) node [below] {$0$};

\draw [dashed] (0,0) circle (1.82cm);

\node [above] at (0,.8) {I};

\node [below] at (0,-.8) {II};

\end{tikzpicture}

\caption{The contour near $z=0$.}
\label{figure-0-1}
\end{figure}
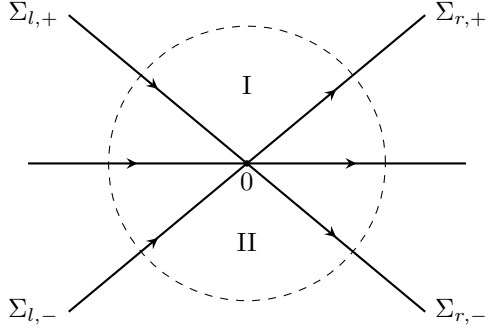

A solution is readily found as 
\begin{equation}\label{V-loc-0}
 V_{loc}^0(z):=
\left\{\begin{array}{lll}
 N(z)\left(
                              \begin{array}{cc}
                                 1 & 0 \\
                                 -(4-z^2)^{1-2\alpha}e^{-L+2n\phi} & 1 \\
                                 \end{array}
                             \right) & \mbox{for} ~z\in \mbox{I},\\
N(z) \left(
                              \begin{array}{cc}
                                 1 & 0 \\
                                 (4-z^2)^{1-2\alpha}e^{-L+2n\phi} & 1 \\
                                 \end{array}
                             \right) & \mbox{for} ~z\in \mbox{II},\\
N(z) & \mbox{for} ~z\in U(0,\varepsilon)\setminus (\mbox{I}\cup \mbox{II}),
\end{array}\right.
\end{equation}
where I and II are depicted in Figure \ref{figure-0-1}.

It is worth noting that 
\begin{equation}\label{V-0-jump-error}
 J_{V_{loc}^0}(z)J_V^{-1}(z) =I+O (z^2/n),
\end{equation}uniformly for $z\in  \{ U(0,\varepsilon)\cap \Sigma\} \cup \partial U(0,\varepsilon)$. 
Following the discussion in \cite[Section 7]{DeiftZhouS}, 
from \eqref{V-0-jump-error}  we conclude that
\begin{equation}\label{V-0-V-loc}
V_0(z)=V_{loc}^0(z)\left (I+O (1/n)\right ),
\end{equation}uniformly for $z\in  U(0,\varepsilon)$.

\begin{remark}
 An alternative  derivation of \eqref{V-0-V-loc} is given as follows. Let $H(z)=V_0(z) \left ( V_{loc}^0\right )^{-1} (z)$ for $z\in \overline {U(0, \varepsilon)}\; \backslash \Sigma$. 
 It is readily seen from \eqref{V-0-jump-error} that
the jump for $H$ fulfills  $J_H(z)=I+O(1/n)$ for $z\in U(0, \varepsilon) \cap\Sigma$, and that $H(z)=I+O(1/n)$ on $\partial U(0, \varepsilon)$. 
Using the single-layer potential techniques, we seek $H(z)$ of the form
\begin{equation*}
 H(z) =I +\frac 1 {2\pi i}\int_{\Sigma_\varepsilon}\frac {\mu_H(\tau) d\tau}{\tau-z},
 \quad z \in U(0, \varepsilon)\backslash  \Sigma,
\end{equation*}where $\Sigma_\varepsilon=U(0, \varepsilon) \cap\Sigma$.
  The RHP for $H(z)$ is then reduced to the singular integral equation for the new unknown matrix function $\mu_H(\tau)$:
\begin{equation*}
    \mu_H(\tau)= \left [ -\frac 1 2\mu_H(\tau)+ \frac 1 {2\pi i}\;p.\,v. \int_{\Sigma_\varepsilon}\frac {\mu_H(t) dt}{t-\tau}\right ]\left (J_H(\tau)-I\right )+\left (J_H(\tau)-I\right ),
    \quad \tau \in \Sigma_\varepsilon.
\end{equation*}
It is seen that the operator in the square brackets are bounded operator in $L_2(\Sigma_\varepsilon)$. Hence, for large $n$, bearing in mind that $J_H(\tau)-I=O(1/n)$, we see that  the nonhomogeneous equation is contractive. A unique solution is then determined such that $\mu_H(\tau)=O(1/n)$ on $\Sigma_\varepsilon$.
Accordingly, it is readily verified that  $H(z)=I+O(1/n)$ in $\overline {U(0, \varepsilon)}$, and \eqref{V-0-V-loc} follows.
\end{remark}

It is worth mentioning that refinements to \eqref{V-0-V-loc} can be obtained by picking up the later terms in \eqref{V-0-jump-error}.

\subsection{The transformation $V\to S$}

Let
\begin{equation}\label{V-tilde}
\widetilde V(z):=\left\{\begin{array}{llll}
V_l(z), &  z\in U(-2,\varepsilon),\\
V_0(z), & z\in U(0,\varepsilon),\\
V_r(z), & z\in U(2,\varepsilon),\\
N(z), &    \mbox{otherwise}.
\end{array}\right.
\end{equation}
It is readily verified that the matrix-valued function
\begin{equation}
S(z):=V(z)\widetilde V^{-1}(z)
\label{S}
\end{equation}
is a solution of the RHP:
\begin{description}
  \item[($S_1$)] $S(z)$ is analytic in $\mathbb{C}\backslash \Sigma_S $ (see Figure \ref{figure-S});
\item[($S_2$)]    For $z\in \Sigma_S$,
\begin{equation}\label{S-jump}
 S_+(z)=S_-(z) J_S(z),
\end{equation}
where
\begin{equation}
 J_S(z)
	=\widetilde V_-(z)
	J_V(z) \widetilde V_+^{-1}(z)
	=
\left\{\begin{array}{lll}
V_l(z) N^{-1}(z), &  z\in \partial U(-2,\varepsilon),\\
V_0(z)N^{-1}(z), &  z\in \partial U(0,\varepsilon),\\
V_r(z)N^{-1}(z), &  z\in \partial U(2,\varepsilon),\\
N(z) J_V(z) N^{-1}(z), & \mbox{otherwise};
\end{array}\right.
\end{equation}
\item[($S_3$)]   
 as $z\to\infty$,
  \begin{equation}
  S(z)=I+O(1/z).
  \end{equation}
\end{description}

It follows from the matching condition of the local parametrices and the definition of
$\phi$ that
\begin{equation}
	J_S(z)=
	\begin{cases}
	I+ O(1/n),
	& z\in \partial ( U(0,\varepsilon) \cup  U(-2,\varepsilon)\cup U(2,\varepsilon))
	\cup (-2+\varepsilon,-\varepsilon) \cup (\varepsilon,2-\varepsilon),
	  \\
	  I+ O(e^{-cn}),
	& \mbox{ otherwise}
	  ,
\end{cases}
\end{equation}
where $c$ is a positive constant, and the error term is uniform for $z\in \Sigma_S$. Hence, we have
\begin{equation}\label{S-temp-1}
	\Vert
	J_S(z) -I
	 \Vert_{ L^2\cap L^\infty (\Sigma_S
	 ) }
	 =  O(1/n).
\end{equation}
Then,
applying the standard procedure of norm estimation of Cauchy operators $C$ and $C_-$ (cf. \cite{DeiftZhouU}, \cite{DeiftZhouS}), it follows that
\begin{equation}\label{S-1}
S(z)=I+C(\Delta_S+u_S\Delta_S),
\end{equation}
where
$
\Delta_S := J_S(z)-I
$
and
$
u_S:= (I-C_{\Delta_S })^{-1}(C_-\Delta_S)
$.
Furthermore,
by using the technique of deformation of contours and \eqref{S-temp-1}, we have
\begin{equation}\label{S-inf}
S(z)=I+O\left( 1/{n}\right)
\end{equation}
uniformly for $z\in\mathbb{C}$.

\begin{figure}
\centering
\begin{tikzpicture}
\def\rad{1.8cm}
  \path [thick,draw=black,postaction={on each segment={mid arrow=black}}]
  (-3*\rad,0)--(-1.7*\rad,0)
  (-1.7*\rad,0)--(0,0)
  (0,0)--(1.7*\rad,0)
  (1.7*\rad,0)--(3*\rad,0)

  (-3*\rad,0)--(-3*\rad,\rad)
  (-3*\rad,\rad)--(-1.7*\rad,\rad)
  (-1.7*\rad,\rad)--(-\rad,\rad)
  (-1.7*\rad,\rad) --(-1.7*\rad,0)
  (-\rad,\rad) --(0,0)

  (-3*\rad,0)--(-3*\rad,-\rad)
  (-3*\rad,-\rad)--(-1.7*\rad,-\rad)
  (-1.7*\rad,-\rad)--(-\rad,-\rad)
  (-1.7*\rad,-\rad)--(-1.7*\rad,0)
  (-\rad,-\rad) --(0,0)

  (3*\rad,\rad) --(3*\rad,0)
  (1.7*\rad,\rad)--(3*\rad,\rad)
  (\rad,\rad)--(1.7*\rad,\rad)
  (1.7*\rad,\rad)--(1.7*\rad,0)
  (0,0)-- (\rad,\rad)

   (3*\rad,-\rad) --(3*\rad,0)
  (1.7*\rad,-\rad)--(3*\rad,-\rad)
  (\rad,-\rad)--(1.7*\rad,-\rad)
  (1.7*\rad,-\rad)--(1.7*\rad,0)
  (0,0)--(\rad,-\rad)

  ;
 \filldraw (-3*\rad,0) circle (1pt)
 node [left] {$-k_n$};

\filldraw [white] (1.7*\rad,0) circle (.6);

\filldraw [white] (-1.7*\rad,0) circle (.6)
 node [below] {$0$};

\draw [thick] (-1.7*\rad,0) circle (.6);

\draw [thick] (1.7*\rad,0) circle (.6);

 \filldraw (3*\rad,0) circle (1pt)
 node [right] {$k_n$};

\draw [->,>= stealth,thick]  ($(-1.7*\rad,0)+(130:0.6)$)
--($(-1.7*\rad,0)+(120:0.6)$);

\draw [->,>= stealth,thick]  ($(1.7*\rad,0)+(60:0.6)$)
--($(1.7*\rad,0)+(50:0.6)$);

\node [left] at (-3*\rad,\rad) {$-k_n+i\delta$};

\node [left] at (-3*\rad,-\rad) {$-k_n-i\delta$};

\filldraw [white] (0,0) circle (.6);

\draw [thick] (0,0) circle (.6);

\draw [->,>= stealth,thick]  ($(0,0)+(90:0.6)$)--($(0,0)+(85:0.6)$);

\filldraw (0,0) circle (1pt)
node [below left] {\small $0$};

\filldraw (-1.7*\rad,0) circle (1pt)
node [below left] {\small $-2$};

\filldraw (1.7*\rad,0) circle (1pt)
node [below right] {\small $2$}
;

\end{tikzpicture}

\caption{The contour $\Sigma_S$.}
\label{figure-S}
\end{figure}
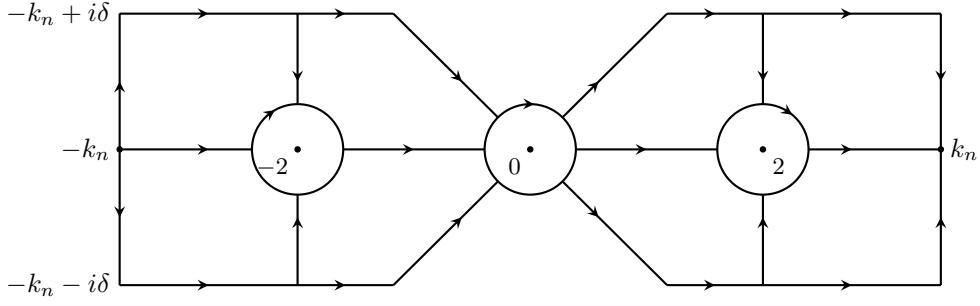

\section{Proof of Theorem \ref{theorem-main}}

Since $V(z)=S(z)\widetilde V(z)$ by \eqref{S},  it follows that
\begin{equation}\label{V-11}
	V_{11}(z)= S_{11}(z)\widetilde V_{11}(z)
	+ S_{12}(z) \widetilde V_{21}(z)
\end{equation}
and
\begin{equation}\label{V-12}
	V_{12}(z)= S_{11}(z)\widetilde V_{12}(z)
	+ S_{12}(z) \widetilde V_{22}(z).
\end{equation}
Moreover, from \eqref{S-inf} it follows that
\begin{equation}\label{S-n}
	S_{11}(z)= 1+O(1/n)
	\qquad \mbox{and}
	\qquad
	S_{12}(z)=O(1/n).
\end{equation}

\noindent\textbf{Region $A_\delta\setminus U(0, \varepsilon)$.} At first, let us consider $z\in A_\delta$  while $z\not \in U(0, \varepsilon)$.
Note by Theorem \ref{Theorem-Y} and \eqref{Y to U} that
\begin{equation}
\pi_n (n^{-1/2}z)=n^{-n/2}U_{11}(z).
\label{pi-z-U}
\end{equation}
Recalling the definition of $\widetilde V(z)$ in \eqref{V-tilde}, we have from
 \eqref{T}, \eqref{Q}, \eqref{V} and \eqref{S-n}
\begin{equation}\label{p-temp-1}
\begin{aligned}
	\pi_n(n^{-1/2}  z)
	&=e^{n/2}  n^{1/2-n/2-\alpha} V_{11}(z)
	(D(z) E(z) )^{-1}  e^{-n\phi }
	\\
	&=e^{n/2}  n^{1/2-n/2-\alpha}
	(D(z) E(z))^{-1}  e^{-n\phi }
	N_{11}
	(1+O(1/n)).
\end{aligned}
\end{equation}
A combination of this with $E(z)=\widetilde E(z)
 e^{-\pi i (1/2-\alpha)}
$,  \eqref{E-tilde} and \eqref{N}
 gives \eqref{A_delta}.

\noindent\textbf{Region $B_\delta\setminus U(0, \varepsilon)$.}
 From \eqref{V}, \eqref{Q}, \eqref{T} and \eqref{R}, it follows that
 \begin{equation}\label{p-temp-4}
 U(z)  =  e^{n/2 \sigma_3} n^{ (1/2-\alpha) \sigma_3 } V(z) [E(z) \widetilde D(z) ]^{-\sigma_3}
  [ \mathcal{V}_-(z)] ^{-1}  e^{-n \phi \sigma_3 }
   [\mathcal{U}_-(z) ]^{-1}
   \qquad \mbox{for }z\in \Omega_{r,+}.
\end{equation}
Then,  on account of \eqref{V-tilde}  we have
\begin{equation}\label{p-temp-2}
\begin{aligned}
 e^{-n/2} n^{ n/2+\alpha-1/2 }  \pi_n( n^{-1/2}z )
 &=  [  e^{-n\phi} N_{11}/(E \widetilde D)
 + e^{n\phi} E \widetilde D/(2i \gamma w
 ) N_{12} ]  S_{11}(z)
 \\
 &\quad +[  e^{-n\phi} N_{21}/(E \widetilde D)
 + e^{n\phi} E \widetilde D/(2i \gamma w
 ) N_{22} ]  S_{12}(z).
\end{aligned}
\end{equation}
From \eqref{D-tilde-asy}, we note that $\widetilde D(z)=1+O(1/n)$ for $z\in B_\delta$, which may be neglected. Hence,  we obtain \eqref{B_delta} by \eqref{E}, \eqref{N}
and \eqref{S-n}.

\noindent\textbf{Region $C_{\delta,1}\cup C_{\delta,2}$.}
Next, let us now consider $z$ in the Airy region.
For simplicity, we only consider the case $z\in C_{\delta,1}$; that is, $z\in \Sigma_{r,+}$ and $\arg \widetilde f_n(z) \in (\pi/2, \pi)$.
Again by \eqref{p-temp-4} and \eqref{V-tilde}, we note that
\begin{equation}
\begin{aligned}
 e^{-n/2} n^{ n/2+\alpha-1/2 }  \pi_n( n^{-1/2}z )
 &=  [  (V_r(z))_{11}S_{11}(z)
    +  (V_r(z))_{21}S_{12}(z)
 ] e^{-n\phi} /(E \widetilde D)
 \\
 &\quad +[  (V_r(z))_{12}S_{11}(z)
    +  (V_r(z))_{22}S_{12}(z)  ]
    e^{n\phi} E \widetilde D/(2 i \gamma w)
    .
\end{aligned}
\end{equation}
Recalling the well-known formula of the Airy functions \cite[(9.2.11)]{NIST}, one can see that
\begin{equation}\label{p-temp-3}
 \sigma_2  \Psi ( \widetilde f_n(z) )
  \sigma_2  =
  \begin{pmatrix}
   \Ai'( \widetilde f_n(z) )  &   \Bi'(\widetilde f_n(z))
   \\
   i \Ai( \widetilde f_n(z) )  &   i\Bi(\widetilde f_n(z))
  \end{pmatrix}
  \begin{pmatrix}
   -i/2 &    -i/2 \\
   1/2  &  -1/2
  \end{pmatrix}
 \qquad  \mbox{for } \arg \widetilde f_n(z) \in (\pi/2, \pi).
\end{equation}
Inserting \eqref{p-temp-3} into \eqref{V-2}, and combining  \eqref{phi-tilde-phi},
\eqref{E-tilde}, \eqref{w-tilde} and \eqref{S-n} yields \eqref{C_delta} for $z\in C_{\delta,1}$.
Following the same argument as given above, one can establish \eqref{C_delta} for $z\in C_{\delta,2}$.

\noindent\textbf{Region $D_\delta$.}
Similarly,  it can be shown that
\begin{equation}
 U(z)  =  e^{n/2 \sigma_3} n^{ (1/2-\alpha) \sigma_3 }  V(z)  [E(z) \widetilde D(z) ]^{-\sigma_3}
   [\mathcal{V}_+^\triangle (z)]^{-1}  e^{-n \phi \sigma_3 }
   [\mathcal{U}_-^\triangle (z)]^{-1}
   \qquad \mbox{for }
   z\in \Sigma_{r,+}^\triangle.
\end{equation}
Again by \eqref{V-tilde}, we have
\begin{equation}
\begin{aligned}
 e^{-n/2} n^{ n/2+\alpha-1/2 }  \pi_n( n^{-1/2}z )
 &=  [ S_{11}(z)N_{11}(z)+ S_{12}(z) N_{21}(z) ]  \frac{2 i \gamma(z) \Pi(z)}{  \widetilde D(z) E(z)}
 e^{-i \theta(z) }  e^{ -n\phi(z) }
 \\
 &\quad +[ S_{11}(z)N_{12}(z)+ S_{12}(z) N_{22}(z) ]  \frac{  \widetilde D(z) E(z)}{2 i \gamma(z) w(z)}
    e^{ n\phi(z) }.
\end{aligned}
\end{equation}
Note that $2 i \gamma(z) \Pi(z)e^{-i \theta(z) }/  \widetilde D(z) = 1/D(z)$ and
$e^{n\phi(z)}$ is exponentially small as $n\to\infty$. From \eqref{E} and \eqref{S-n}, we obtain \eqref{D_delta}.
This completes the proof of Theorem \ref{theorem-main}.

\section{Proof of Theorem \ref{Theorem-origin}}
The proof is similar to that of Theorem \ref{theorem-main}.  For instance, as $z\in A_\delta \cap U(0, \varepsilon)$, \eqref{V-11} and the first equality of \eqref{p-temp-1} are still valid. We need only to replace $V_0(z)$ with  $V_{loc}^0(z)$ up to an error of order $O(1/n)$. Then \eqref{asy-origin} follows accordingly by using the  Stirling's formula. 

The other cases, such as $z\in B_\delta \cap U(0, \varepsilon)$ can be justified similarly by tracing back all the transformations.   Again \eqref{asy-origin} follows  from a  simplification of 
\eqref{p-temp-2}. Here use has been made of the relations between the auxiliary functions, established in Section 4.

\section{Discussion and comparison with known results}
The Riemann-Hilbert approach has proven an powerful tool in dealing with uniform asymptotics. An example of its powerfulness is the uniform approximation in a neighborhood of the origin.
We given in Theorem \ref{Theorem-origin} the leading behavior. Refinements  are  achievable if we pick more terms for the matrix functions $S(z)$ and $V_0(z)$.

To conclude this paper, we will do a insistence check by comparing  our formulas in Theorem \ref{theorem-main} with the results previously obtained   in \cite{GohWimpO} and \cite{LeeWongU}. More precisely,  we check  the approximations  in regions   disjointing  the support of the equilibrium measure,  with  Goh and  Wimp \cite{GohWimpO}. Comparison will also be made with Lee and Wong  \cite{LeeWongU} on uniform asymptotics at a tuning point, corresponding to the  band-saturated region endpoint $x=2$ in the present paper.

   The first of them is taken from   \cite{GohWimpO}, obtained by using Darboux's method.
In the notations of this paper, it reads
\begin{equation}\label{GohWimp-1}
	f_n^{(\alpha)}(x)
	\sim \frac{ x^n n^{\alpha-1/x^2-1}  e^{1/x^2} }
	{ \Gamma(\alpha-1/x^2) }
\end{equation}
uniformly for $x$ in compact subset $K\subset \mathbb{C}\setminus [-1,1]$. We next derive from \eqref{D_delta} the formula for $f_n^{(\alpha)}(x)$ when $x$ is fixed (i.e., $z=O(n^{1/2})$). For $x>0$, i.e., $z>k_n$,
substituting $z=n^{1/2}x$ into \eqref{D} gives
\begin{equation}
D_+(x)=\frac{\Gamma(\alpha-1/x^2)}{\sqrt{2\pi}e^{1/x^2}}
\left(\frac{1}{x^2}\right)^{1/x^2+1/2-\alpha}e^{(1/x^2+1/2-\alpha)\pi i}.
\end{equation}
From \eqref{D_delta} and \eqref{phi-tilde-phi},
it follows that
\begin{align}
\pi_n(x)\sim\frac{\Gamma(\alpha)n^{1/2-n/2-\alpha}e^{n/2+1/x^2-n\widetilde\phi(x)}}{\Gamma(\alpha-1/x^2)(n x^2-4)^{1/4}}
\left(\frac{1}{x^2}\right)^{-1/x^2-1/2+\alpha}
\left(\frac{n^{1/2}x+\sqrt{nx^2-4}}{2}\right)^{2\alpha-1/2}.
\label{pi-compare-1}
\end{align}
Moreover, note  by \eqref{phi-tilde-exact} that
$
\widetilde\phi(n^{1/2}x)\sim  1/2-\log (n^{1/2} x)+2\log (n^{1/2}x) /(n x^2)
$ as $n\to\infty$. Thus, we have
\begin{equation}
	\pi_n(x)\sim \frac{ \Gamma(\alpha) n^{-1/x^2}e^{1/x^2}  x^n }
	{ \Gamma(\alpha -1/x^2) }.
\end{equation}
In view of \eqref{leading}, it is readily seen that \eqref{D_delta} agrees with \eqref{GohWimp-1}.

Next, we consider the case when $z=2+O(n^{-2/3})$. First, we
introduce the notation $t=(\nu/n)^{1/2}z$, where $\nu=n+2\alpha-1/2$.
From \eqref{f-tilde}, it can be shown that
\begin{equation}
\left[
\varphi(z/2)^{2\alpha-1/2}
-\varphi(z/2)^{-2\alpha+1/2}
\right](z^2-4)^{-1/4}(\widetilde f_n(z))^{-1/4}=O(\nu^{-1/6})
\end{equation}
and
\begin{equation}
\left[
\varphi(z/2)^{2\alpha-1/2}
+\varphi(z/2)^{-2\alpha+1/2}
\right](z^2-4)^{-1/4}(\widetilde f_n(z))^{1/4}=\sqrt{2}
\nu^{1/6}(1+O(\nu^{-2/3}))
.
\end{equation}
Furthermore, we have $\widetilde f_n(z)=\nu^{2/3}[\zeta(t)+O(1/\nu) ]$, where  $\zeta(t)$ is the function introduced by Lee and Wong \cite[(2.15)]{LeeWongU}. From \eqref{C_delta} and \eqref{leading}, it follows that
\begin{equation}
\begin{aligned}
f_n^{(\alpha)}(\nu^{-1/2} t)&=
n^{-1/3-n/2}e^{n/2}
\left\{\cos(\pi\alpha-\nu \pi/t^2)\left[\text{Ai}(\nu^{2/3}\zeta(t))+O\left(\nu^{-1/3}\right)\right]
\right.\\
&~~~+\left.\sin(\pi\alpha-\nu\pi/t^2)\left[\text{Bi}(\nu^{2/3}\zeta(t))+O\left({\nu^{-1/3}}\right)\right]
\right\},
\label{pi Lee}
\end{aligned}
\end{equation}
which exactly agrees with (3.7) in Lee and Wong \cite{LeeWongU} if we note that $(4\zeta(t)/(t^2-4))^{1/4}{= 1+O(\nu^{-2/3})}$.

\section*{Acknowledgements}
 The work of Shuai-Xia Xu  was supported in part by the National
Natural Science Foundation of China under grant number
11201493, GuangDong Natural Science Foundation under grant number S2012040007824, and the Fundamental Research Funds for the Central Universities under grand number 13lgpy41.
 Yu-Qiu Zhao  was supported in part by the National
Natural Science Foundation of China under grant numbers 10471154 and
10871212.

\end{document}